\documentclass[hidelinks,onefignum,onetabnum]{siamart250211} 

\usepackage{lipsum}
\usepackage{amsfonts}
\usepackage{graphicx}
\usepackage{epstopdf} 
\usepackage{algorithm}
\usepackage{algorithmicx}
\usepackage{algpseudocode}
 
\ifpdf
  \DeclareGraphicsExtensions{.eps,.pdf,.png,.jpg}
\else
  \DeclareGraphicsExtensions{.eps}
\fi

\newtheorem{example}{Example}
\usepackage{multirow}
\usepackage{tikz}

\newsiamremark{remark}{Remark}
\newsiamremark{hypothesis}{Hypothesis}
\crefname{hypothesis}{Hypothesis}{Hypotheses}
\newsiamthm{claim}{Claim}

\headers{Convergence analysis of SPH method}{ZHONGHUA QIAO AND YIFAN WEI}

\title{Convergence analysis of SPH method on irregular particle distributions for the Poisson equation \thanks{Submitted to the editors DATE.
\funding{This work is supported by the CAS AMSS-PolyU Joint Laboratory of Applied Mathematics (No. JLFS/P-501/24).
The first author was partially supported by the Hong Kong Research Grants Council RFS grant RFS2021-5S03, GRF grants 15302122 and 15305624.
The second author was partially supported by the Hong Kong Polytechnic University Postdoctoral Research Fund 1-W30N.
}}}

\author{ZHONGHUA QIAO \thanks{Department of Applied Mathematics, The Hong Kong Polytechnic University,  Hung Hom, Kowloon, Hong Kong
  (\email{zhonghua.qiao@polyu.edu.hk}).}
\and YIFAN WEI \thanks{Corresponding author. Department of Applied Mathematics, The Hong Kong Polytechnic University,  Hung Hom, Kowloon, Hong Kong
  (\email{yi-fan.wei@polyu.edu.hk}).}}

\usepackage{amsopn}

\usepackage{amsmath}
\begin{document}

\maketitle

\begin{abstract}
The numerical accuracy of particle-based approximations in Smoothed Particle Hydrodynamics (SPH) is significantly affected by the spatial uniformity of particle distributions, especially for second-order derivatives. This study aims to enhance the accuracy of SPH method and analyze its convergence with irregular particle distributions.
By establishing regularity conditions for particle distributions, we ensure that the local truncation error of traditional SPH formulations, including first and second derivatives, achieves second-order accuracy.
Our proposed method, the volume reconstruction SPH method, guarantees these regularity  conditions while preserving the discrete maximum principle. Benefiting from the discrete maximum principle, we conduct a rigorous global error analysis in the $L^\infty$-norm for the Poisson equation with variable coefficients, achieving second-order convergence. Numerical examples are presented to validate the theoretical findings.
\end{abstract}



\section{Introduction}
The Smoothed Particle Hydrodynamics (SPH) method, originally introduced by Lucy \cite{lucy1977numerical} and independently by Gingold and Monaghan \cite{gingold1977smoothed}, is a mesh-free Lagrangian particle approach with significant potential for addressing deformation and free-interface problems.
Owing to its Lagrangian characteristics, SPH inherently tracks free surfaces or interfaces in complex fluid simulations \cite{lind2020review}. 
SPH has also demonstrated considerable potential in handling multiphase flows with high density ratios \cite{lind2016incompressible}. 
In solid mechanics, the unique properties of solids\textemdash such as large deformations and diverse material constitutive laws, including damage and crack propagation\textemdash align well with SPH formulations \cite{meng2021advances}. The particle-based modeling in SPH avoids the common issue of mesh distortion found in mesh-based methods, making it naturally suited for problems involving large deformations. 
SPH more naturally ensures the conservation of mass and momentum \cite{feng2023energy,zhu2022energy}. 
Additionally, the computational efficiency of SPH has been significantly enhanced through the utilization of GPU acceleration \cite{chow2018incompressible}. 
These advantages have contributed to the widespread adoption of SPH in fields like astrophysics, engineering, and computer graphics.

On the other hand, the traditional SPH method still suffers from the inadequacy of low accuracy when dealing with irregular particle distributions.
Specifically, the error in SPH arises from two primary sources: smoothing error and discretization error \cite{quinlan2006truncation} (or integration error \cite{lind2016high,vacondio2021grand}). 
Smoothing error arises from the kernel-based smoothing operation \cite{LITVINOV2015394}.
The convergence accuracy of this error depends on the chosen kernel function. 
Specifically, Lind et al. \cite{lind2016high} achieved enhanced accuracy by employing high-order Gaussian kernels on regular particle distributions for transient flow simulations. Furthermore, Wang et al. \cite{wang2024fourth} present a new  4th-order truncated Laguerre-Gauss kernel yields substantially smaller relaxation residues than both Wendland and Laguerre-Wendland kernels. 
When a positive smoothing function is used to ensure a meaningful or stable representation of physical phenomena, the highest level of precision attainable for approximating functions is of the second order \cite{liu2010smoothed}. 
The discretization error originates from the particle-based discretization of the domain.
This error depends mainly on particle distribution regularity under the  traditional definition of particle volume. Indeed, Quinlan et al. \cite{quinlan2006truncation} demonstrated that with a constant ratio of particle spacing to smoothing length, non-uniform particle distributions can lead to divergent behavior.
Specific physical governing equations can promote relatively uniform particle distributions, thereby enhancing the convergence accuracy of the SPH method. For instance, Litvinov et al. \cite{LITVINOV2015394} demonstrated that a constant pressure field can transform randomly distributed particles into a configuration satisfying the partition of unity condition. This resulting particle distribution exhibits enhanced uniformity, which is essential for achieving zeroth-order consistency in SPH approximations.
However, in other important physical models, the particle distribution in the SPH method may exhibit irregularities. For instance, unphysical void regions may occur in fluids with high Reynolds numbers\cite{lee2008comparisons}. Additionally, tensile instability can lead to unphysical particle motion, resulting in the formation of particle clusters \cite{bagheri2024review}. Therefore, there is a need to improve methods for effectively handling irregular particle distributions.

Improving the consistency of the SPH approximation for irregular particles has naturally received significant attention and has made substantial progress.
A popular strategy is to develop improved SPH schemes using Taylor series expansions of functions and their derivatives. Examples include the corrective smoothed particle method (CSPM) by Chen and Beraun \cite{chen_corrective_1999}, and the finite particle method (FPM) by Liu et al. \cite{liu_restoring_2006,liu_modeling_2005}.
The possible numerical errors existent in a lower order derivative in CSPH may carry over to higher order derivatives, which does not happen in FPM.
Regarding the restoration of the SPH approximate consistency, the FPM exhibits first-order consistency if only first-order derivatives are retained in the Taylor series expansion. 
Note that both CSPM and FPM require solving systems of linear equations with potentially ill-conditioned corrective matrices, particularly when handling highly irregular particle distributions \cite{zhang2018meshfree}. This may lead to numerical instabilities or premature termination of simulations.  As an improvement, Zhang et al. \cite{zhang2018decoupled} proposed the decoupled finite particle method (DFPM), which eliminates the need for solving matrix equations while maintaining superior accuracy compared to traditional SPH methods. 
Recent advances have further enhanced consistency in particle methods through alternative approaches. Nasar et al. \cite{nasar2021high-2} developed a high-order consistency correction scheme based on modified SPH and FPM, significantly improving the order of discretization error. Meanwhile, Zhang et al. \cite{zhang2025towards} introduced the reverse KGC formulation, which achieves high-order consistency while maintaining conservation properties.
In addition, Yang et al. \cite{yang_improved_2024}  reformulated the fundamental equations of the traditional FPM on the basis of matrix decomposition, and developed the generalized finite particle method (GFPM) which can be theoretically proven to be always stable.
The GFPM, however, requires additional constraints on the model or the utilization of other numerical methods to obtain the value of second derivatives.
Based on Taylor series expansion, Fatehi et al. \cite{fatehi2011error} investigated a variety of derivative approximation schemes for regular and irregular particle distributions. They introduced a novel SPH scheme aimed at approximating second derivatives while maintaining the first-order consistency property.
Recently, Lian et al. \cite{lian2021general} developed a  general
SPH scheme for second derivatives to improve the accuracy of solving anisotropic diffusion unsaturated permeability problems.
We emphasize that the first-order consistent SPH approximation can only guarantee  first-order accuracy in approximating first derivatives. We will discuss this in Section \ref{Sec2}, using FPM as an example. Given that the smoothing error is of second-order, it becomes important to investigate methods to attain second-order accuracy in particle approximation with minimal resources, particularly for second derivatives.

Despite practical advancements in improving the consistency of the SPH approximation, there remains a lack of rigorous mathematical analysis for convergence under realistic conditions, particularly for irregular particle distributions.
In fact, quantifying the mathematical convergence for SPH is seen as a major challenge \cite{vacondio2021grand}. An important early contribution was made by Musa and Vila \cite{ben2000convergence}, who demonstrated the convergence of their SPH scheme for scalar nonlinear conservation laws in one dimension.
To further deepen the theoretical understanding of SPH, Du et al. \cite{du2020mathematics} conducted a mathematical analysis of asymptotically compatible discretization for the non-local Stokes equation using the Fourier spectral method.
Lee et al. \cite{lee2019asymptotically} presented nonlocal models for linear advection in one spatial dimension, along with their particle-based numerical discretizations. They pointed out that extending the asymptotically compatible nature to two-dimensional irregular particle distributions poses a significant challenge.
Indeed, analyzing the convergence of the SPH method for irregular two-dimensional particle distributions is quite challenging. Enhancing  the consistency of SPH approximations frequently introduces additional complexity in SPH formulations, making convergence analysis more challenging.
The existing high-accuracy SPH methods based on Taylor series expansion, such as FPM and CSPH, implicitly rely on information from low-order derivatives and function values in their approximation formats for high-order derivatives. 
While these methods are effective for solving evolution equations, this fundamental characteristic renders them poorly suited for boundary-value problems like the Poisson equation, where both the function values and first derivatives are unknown.   
The accurate numerical solution of such Poisson-type equations holds substantial significance in fluid simulation. Addressing this challenge, Nasar et al. \cite{nasar2021high} made significant advances by developing a high-order accurate Eulerian incompressible SPH, which provides rigorous treatment of the pressure Poisson equation  while precisely enforcing both the velocity Dirichlet boundary conditions and pressure Neumann boundary conditions. 
To address the existing gap in the mathematical analysis of SPH methods for irregular particle distributions, the present study focuses on analyzing the convergence of the variable coefficient Poisson problem across non-uniform particle distributions \cite{min2006supra}:
\begin{align}\label{eq1.1}
 \mathcal{L}u (\mathbf{x}):=\nabla\cdot(a(\mathbf{x}) \nabla u (\mathbf{x}))=f(\mathbf{x}), \quad \mathbf{x}\in \Omega,
\end{align}
subject to  the Dirichlet boundary condition
\begin{align}\label{Dirichlet}
  u(\mathbf{x})&=0, \quad\forall \mathbf{x}\in\partial\Omega.
\end{align}
Here, $u (\mathbf{x})$ is the unknown function,  and the variable coefficient $a$ is positive and bounded from below by $a_{\min}>0$. $\Omega=[0,l]^d$ is a hypercube domain in $\mathbb{R}^d$.

In this study, we introduce a refined SPH technique, which we call the Volume Reconstruction SPH (VRSPH) method, to solve equation \eqref{eq1.1} and we examine the convergence of the numerical approach.
The primary innovation of this research involves the application of volume reconstruction to ensure the regularity conditions and uphold the discrete maximum principle. The implementation of these conditions ensures the second-order accuracy in the approximation of the Laplace operator, while the adherence to the discrete maximum principle lays a robust groundwork for the convergence analysis. The regularity conditions we discuss are intricately linked to the discrete consistency conditions outlined in \cite{liu_restoring_2006}, which are essential for managing irregular particle distributions.
The development of this methodological framework offers novel tools and perspectives for the analysis of steady-state partial differential equations that utilize the Laplace operator.
We summarize our main contributions as follows:
\begin{enumerate}
\item We show that traditional SPH can reach second-order accuracy under certain regularity conditions of the particle distribution.
\item The VRSPH method is introduced to handle irregular particle distributions, with the goal of improving the accuracy of the second derivative to second order.
\item We present the first $L^\infty$-error estimation for the SPH method on irregular particle distributions in solving the variable coefficient Poisson equation, successfully attaining second-order accuracy.
\end{enumerate}

The rest of this paper is organized as follows.
Section \ref{Sec2} discusses the fundamentals of traditional SPH and the FPM method. Section \ref{Sec3} details the regularity conditions necessary for particle distributions to achieve second-order approximation accuracy. The VRSPH method is introduced in Section \ref{Sec4}. Section \ref{Sec5} provides an error estimation for the VRSPH method when applied to the variable coefficient Poisson equation. Section \ref{Sec6} presents numerical experiments that validate the theoretical claims. The paper concludes with Section \ref{Sec7}.

\section{Traditional SPH and FPM}\label{Sec2}
In this part, we will offer a concise introduction to the traditional SPH technique and the widely adopted variant called the FPM. Both the CSPH method and the KGF method share foundational principles with FPM. However, as CSPH and KGF do not surpass FPM in accuracy, they will not be the central subjects of our discussion here.
\subsection{Brief review of traditional SPH}
By the definition of  Dirac delta function $\delta$, there holds
\begin{equation}\label{eq1}
  f(\mathbf{x})=\int_\Omega f(\mathbf{x}')\delta(\mathbf{x}-\mathbf{x}')\textrm{d} \mathbf{x}'.
\end{equation}
In  SPH, $\delta(\mathbf{x}-\mathbf{x}')$ is replaced by a smoothing function $W(\mathbf{x}-\mathbf{x}',h)$ which usually satisfies the following conditions
\begin{itemize}
    \item Normalization condition:$\int_\Omega W(\mathbf{x}-\mathbf{x}',h)\textrm{d}\mathbf{x}'=1$;
    \item Symmetric property: $W(\mathbf{x}-\mathbf{x}',h)=W(\mathbf{x}'-\mathbf{x},h)$;
    \item Compact condition: $W(\mathbf{x}-\mathbf{x}',h)=0 \text{ where } \|\mathbf{x}-\mathbf{x}'\|\geq h$;
    \item Decay condition: $ W(\mathbf{x},h)\leq W(\mathbf{x}',h)$ where $\|\mathbf{x}\|\geq \|\mathbf{x}'\|$.
\end{itemize}
Here, $\|\cdot\|$ represents the Euclidean norm and $h$ is the influence radius of kernel.  $W(\mathbf{x},h)$ affects only a specific support domain, similar to the Dirac delta function. Then the \textit{Kernel approximation} or \textit{Integral approximation} of $f(\mathbf{x})$ becomes
\begin{equation}\label{eq2}
f_I(\mathbf{x})=  \int_\Omega f(\mathbf{x}')W(\mathbf{x}-\mathbf{x}',h)d\mathbf{x}'\approx f(\mathbf{x}).
\end{equation}
By replacing $f$ in \eqref{eq2} with its corresponding derivatives and applying the method of integration by parts to transfer the derivatives onto the smoothing function $W$, we can obtain the Kernel approximation of the spatial derivative as follows
\begin{equation}\label{eq3}
[D^\alpha f]_I(\mathbf{x})=  -\int_\Omega f(\mathbf{x}')D^\alpha_{\mathbf{x}'} W(\mathbf{x}-\mathbf{x}',h)d\mathbf{x}',
\end{equation}
where $\alpha$ is a multi-index.

In discrete space, the integration operation of \eqref{eq2} can be replaced by summation.
Denote $f_j := f(\mathbf{x}_j)$, $\mathbf{x}_{ij}:=\mathbf{x}_i-\mathbf{x}_j$ and $W_{ij} := W(\mathbf{x}_i-\mathbf{x}_j,h)$.
Then, the \textit{Particle approximation} of the traditional SPH can be expressed in the following form:
$$ \langle f_i\rangle =\sum_j  v_j f_jW_{ij}\approx f(\mathbf{x}_i).$$
Here, $i$ denotes the interpolating particle and $j$ refers to the neighbouring particles within the support. $v_j$ is the volume of the particle $j$ .
Similarly,  the Particle approximation of derivatives can be given as
  $$ \langle  D^\alpha f_{i}\rangle = -\sum_j v_j f_j  D^\alpha_{\mathbf{x}'} W_{ij}.$$
Given that $R=\|\mathbf{x}_{ij}\|/h$, the cubic spline kernel can be expressed as follows
\begin{align}\label{cubic-spline}
W_{ij}=\sigma_d\frac{1}{h^{d}}\times
\left\{
\begin{array}{ll}
1-6R^2+6R^3, & \hbox{$0\leq R <0.5;$} \\
2(1-R)^3, & \hbox{$0.5\leq R < 1;$} \\
0, & \hbox{$R\geq 1,$}
\end{array}
\right.
\end{align}
which is a typical choice for the smoothing kernel.
The kernel normalization factors $\sigma_d$ for the respective dimensions $d = 1, 2, 3$ are $\sigma_d= 4/3, 40/(7\pi)$ and $ 8/\pi$.
In general, we can express the kernel function (including the cubic spline kernel) as: $$W_{ij}=\sigma_d\frac{1}{h^{d}}\hat{W}(R),$$
 where $\hat{W}(R)$ and $\sigma_d$ vary with the choice of kernel function.
It can be checked that
\begin{align}\label{estimateofnablaW}
\nabla_{\mathbf{x}_i} W_{ij}=\frac{\sigma_d}{h^{d}}\nabla_{\mathbf{x}_i}\hat{W}(R)=\frac{\sigma_d}{h^{d+1}}\partial_R\hat{W}  \frac{\mathbf{x}_{ij}}{\|\mathbf{x}_{ij}\|}.
\end{align}
By the use of \eqref{estimateofnablaW} and denote $\nabla_i  W_{ij}:=\nabla_{\mathbf{x}_i}  W_{ij}$, we can obtain the following two lemmas.
\begin{lemma}\label{lem1}
\begin{align}
    \|\nabla_i  W_{ij}\|_\infty\leq Ch^{-d-1}.
\end{align}
Here, $C$ is only dependent on the kernel function $\hat{W}(R)$ and the dimension parameter $d$.
\end{lemma}

\begin{lemma}\label{lem1-2}
$\nabla_iW_{ij}$ is parallel to $\mathbf{x}_{ij}$, namely
\begin{align}
 \nabla_iW_{ij}=\Big(\frac{\sigma_d}{h^{d+1}}\frac{\partial_R\hat{W}(R)}{\|\mathbf{x}_{ij}\|} \Big)\mathbf{x}_{ij}.
\end{align}
Furthermore, if  $W(\mathbf{x},h)$ satisfies the decay condition, then $\partial_R\hat{W}(R)\leq 0$ and
\begin{align}
  \mathbf{x}_{ij}\cdot \nabla_iW_{ij} \leq 0.
\end{align}
\end{lemma}
The above-mentioned decay condition signifies that the value of a particle's smoothing function diminishes progressively with increasing distance from the particle. This principle is based on the physical notion that a particle in closer proximity should have a stronger impact on the particle under consideration. Fundamentally, the force of interaction between two particles weakens as their distance grows.

Considering the stability and accuracy of the calculation, the following traditional SPH formulations are generally used in practice:
 \begin{align}
  \text{Divergence} &&\langle \nabla\cdot  \mathbf{A}(\mathbf{x}_i)\rangle&=-\sum_jv_j\mathbf{A}_{ij}\cdot\nabla_iW_{ij},\\
  \text{Gradient} &&\langle\nabla f(\mathbf{x}_i)\rangle&=-\sum_j  v_jf_{ij} \nabla_iW_{ij} \label{gradientscheme},\\
  \text{Laplacian} &&\langle \Delta \mathbf{A}(\mathbf{x}_i)\rangle&= 2\sum_j v_j\mathbf{A}_{ij}\frac{ \mathbf{x}_{ij}\cdot \nabla_iW_{ij} }{\|\mathbf{x}_{ij}\|^2}\label{laplacescheme},\\
  \text{Morris operator}\cite{lind2020review,morris1997modeling} &&\langle \nabla\cdot( a(\mathbf{x}_i)\nabla \mathbf{A}(\mathbf{x}_i))\rangle&= \sum_j v_j\mathbf{A}_{ij}\frac{(a_i+a_j)\mathbf{x}_{ij}\cdot \nabla_iW_{ij} }{\|\mathbf{x}_{ij}\|^2}.\label{laplacescheme2}
\end{align}
Here, $\mathbf{A}_{ij}:=\mathbf{A}_{i}-\mathbf{A}_{j}$  and $f_{ij}:=f_{i}-f_{j}$.
\subsection{Approximating derivatives with FPM}
The FPM ensures $C^1$ particle consistency for interior particles with non-uniform distributions by retaining only the first order derivatives. In this context, $C^1$ consistency refers to the ability to reproduce or approximate a linear function. The fundamental concept behind the proof involves the Taylor expansion, which is briefly demonstrated in the following one-dimensional case.

For any function $f(x)\in C^\infty(\Omega)$, there exists a Taylor expansion as follows
\begin{align}\label{Tylor_linear}
f_{j}=f_i+ f_i' x_{ji}+\frac{1}{2!} f''(\xi_j) x_{ji}^2,
\end{align}
where $\xi_j$ is a point between $x_i$ and $x_j$.
Multiply both sides of equation \eqref{Tylor_linear} with kernel $v_jW_{ij}$ and $v_j\nabla_iW_{ij}$, respectively, and then sum over the index $j$ to get
\begin{align}\label{FPM2}
  &\left[\begin{array}{ll}
  \sum_j v_j  W_{ij}         &\sum_j v_j x_{ji} W_{ij} \\
  \sum_j v_j \nabla_iW_{ij} &\sum_j v_j x_{ji}\nabla_iW_{ij} \\
  \end{array}\right]
  \left[\begin{array}c {f_i}\\ {f_i'}  \end{array}\right]\nonumber\\
  = &\left[\begin{array}l \sum_jv_j f_j  W_{ij}\\ \sum_j v_j f_j  \nabla_iW_{ij} \end{array}\right]
-\frac{1}{2!}\left[\begin{array}l \sum_jv_j f''(\xi_j)  x_{ji}^2W_{ij}\\ \sum_j v_j f''(\xi_j)x_{ji}^2  \nabla_iW_{ij} \end{array}\right].
  \end{align}
The corresponding particle approximations are given by
\begin{align}
\left[\begin{array}c {\langle f_i\rangle}\\ {\langle f_i' \rangle}  \end{array}\right]
=  &\left[\begin{array}{ll}
  \sum_j v_j  W_{ij}         &\sum_j v_j x_{ji} W_{ij} \\
  \sum_j v_j \nabla_iW_{ij} &\sum_j v_j x_{ji}\nabla_iW_{ij} \\
  \end{array}\right]^{-1}
  \left[\begin{array}l \sum_jv_j f_j  W_{ij}\\ \sum_j v_j f_j  \nabla_iW_{ij} \end{array}\right].
  \end{align}
If $f$ is a linear function, then $f''(\xi_j)\equiv0$. Provided that the matrix specified in \eqref{FPM2} is nonsingular, the truncation errors for both the function and its first derivative, expressed as $\|f_i-\langle f_i\rangle\|_\infty$ and $\|f_i'-\langle f_i'\rangle\|_\infty$, respectively,  can be kept within the bounds of machine precision. This supports the claim that FPM guarantees $C^1$ particle consistency.

If $f(x)\in C^\infty(\Omega)$ is not a linear function, the truncation error of FPM depends on the second term on the right-hand side of the system \eqref{FPM2}. It is easy to verify that $\sum_j|v_j|\leq Ch^d$. Then, by using Lemma \ref{lem1}, one can obtain
\begin{align}
\sum_j v_j f''(\xi_j)x_{ji}^2  \nabla_iW_{ij} \leq C h\|f\|_{C^2(\Omega)}.
\end{align}
Given that the infinity norm of the inverse of the matrix given in \eqref{FPM2} is bounded, the truncation error for the first derivative achieves \textbf{first-order} accuracy.

In the context of using the FPM to compute $ f''(x)$, it is necessary to retain the second derivative term in the Taylor series expansion. Therefore, the Taylor expansion can be expressed as follows:
\begin{align} \label{Tylor1D}
f_{j}=f_i+ f_i' x_{ji}+\frac{1}{2!} f_i'' x_{ji}^2 +\frac{1}{3!} f'''(\xi_j) x_{ji}^3,
\end{align}
where $\xi_j$ is a point between $x_i$ and $x_j$.
To proceed with the computation using FPM, we multiply each term of the Taylor expansion \eqref{Tylor1D} by $v_jW_{ij}$, $v_j\nabla_i W_{ij}$ and $v_j\Delta_i W_{ij}$, respectively, and then sum over all $j$. Here $\Delta_iW_{ij}:=-2\mathbf{x}_{ij}\cdot \nabla_iW_{ij} /\|\mathbf{x}_{ij}\|^2$. These operations result in the following system:
  \begin{align}\label{FPM1}
  &\left[\begin{array}{lll}
  \sum_jv_j  W_{ij}         &\sum_j v_j x_{ji} W_{ij}       &\sum_jv_j  x_{ji}^2W_{ij}/2\\
  \sum_j v_j \nabla_iW_{ij} &\sum_j v_j x_{ji}\nabla_iW_{ij}&\sum_j v_j x_{ji}^2\nabla_iW_{ij}/2\\
  \sum_j v_j  \Delta_iW_{ij}&\sum_j v_j x_{ji}\Delta_iW_{ij}  &\sum_j v_j x_{ji}^2\Delta_iW_{ij}/2
  \end{array}\right]
  \left[\begin{array}c {f_i}\\ {f_i'}\\  f_i''  \end{array}\right]\nonumber\\
  =& \left[\begin{array}l \sum_jv_j f_j  W_{ij}\\ \sum_j v_j f_j  \nabla_iW_{ij}\\ \sum_j v_j f_j  \Delta_iW_{ij}\end{array}\right]
-\frac{1}{3!}\left[\begin{array}l \sum_jv_j f'''(\xi_j)  x_{ji}^3W_{ij}\\ \sum_j v_j f'''(\xi_j)  x_{ji}^3\nabla_iW_{ij}\\ \sum_j v_j f'''(\xi_j)  x_{ji}^3\Delta_iW_{ij}\end{array}\right].
  \end{align}
In view of the fact
  \begin{align*}
    \Big|\Delta_iW_{ij}\Big|=\Big|2\frac{ \mathbf{x}_{ij}\cdot \nabla_iW_{ij} }{\|\mathbf{x}_{ij}\|^2}\Big|=O(\frac{1}{\|\mathbf{x}_{ij}\|h^{d+1}}),
  \end{align*}
we have
\begin{align}
\sum_j v_j f'''(\xi_j)x_{ji}^3 \Delta_iW_{ij} \leq C h\|f\|_{C^3(\Omega)}.
\end{align}
Given that the infinity norm of the inverse of the matrix specified in \eqref{FPM2} is bounded, the truncation error of second derivative attains \textbf{first-order} accuracy.

In fact, other high-accuracy correction algorithms for SPH, like CSPH and KGF, encounter a similar challenge: achieving only first-order accuracy in the particle approximation of derivatives.
Additionally, we will carry out numerical validation of the FPM method's accuracy in derivative approximation in Section \ref{Sec6}.

\section{Regularity conditions for particle approximation}\label{Sec3}
In this section, we outline the regular conditions and show that the truncation errors for both first and second derivatives achieve second-order convergence under these specified conditions.

Initially, we define a d-tuple ordered pair $\alpha = (\alpha_{1}, \alpha_{2}, \ldots, \alpha_{d})$ as a multi-index, with each $\alpha_{k}$ being a non-negative integer. Furthermore, we denote the sum of the components as $|\alpha| = \sum_{k=1}^{d} \alpha_{k}$. For two multi-indices $\alpha$ and $\beta$:
\begin{itemize}
  \item The operations $\alpha \pm \beta$ result in the tuple $(\alpha_{1} \pm \beta_1, \alpha_{2} \pm \beta_2, \ldots, \alpha_{d} \pm \beta_d)$.
  \item The comparison $\alpha \leq \beta$ implies $\alpha_i \leq \beta_i$ for each $i$ from 1 to $d$.
\end{itemize}
Consider a $d$-dimensional real-valued function $f \in C^{\infty}(\Omega)$. The $m$-th order partial derivatives $\frac{\partial^m f}{\partial x^{\alpha_1}_1 \cdots \partial x^{\alpha_d}_d}$ with $m = \alpha_1 + \cdots + \alpha_d$, can be compactly expressed in multi-index notation as:
$D^{\alpha}f = D^{\alpha_1}_1 \cdots D^{\alpha_d}_d f$,
where   $D_i^{\alpha_i} := \frac{\partial^{\alpha_i}}{\partial x_i^{\alpha_i}}$ denotes the $\alpha_i$-th order partial derivative along the $i$-th coordinate. For the sake of brevity, we use $C$ to denote a generic positive constant that is independent of $h$.

\begin{lemma}\label{barbeta}
  For any $|\alpha|\geq 3, |\beta|=1, d=1,2,3$, there is a multi-index $\bar{\beta}$ with $|\bar{\beta}|=1$ satisfying
\begin{align}\label{eq1_bar}
\mathbf{x}_{ji}^\alpha \nabla_iW_{ij}^\beta
=\mathbf{x}_{ji}^{\alpha-2\bar{\beta}+\beta}\mathbf{x}_{ji}^{\bar{\beta}}\nabla_iW_{ij}^{\bar{\beta}}.
\end{align}
Furthermore, if  $W(\mathbf{x},h)$ satisfies the decay condition, there holds
\begin{align}\label{eq2_bar}
|\mathbf{x}_{ji}^\alpha \nabla_iW_{ij}^\beta|
=|\mathbf{x}_{ji}^{\alpha-2\bar{\beta}+\beta}|\mathbf{x}_{ji}^{\bar{\beta}}\nabla_iW_{ij}^{\bar{\beta}}.
\end{align}
\end{lemma}
\begin{proof}
By the definition of $\nabla_iW_{ij}$ in \eqref{estimateofnablaW}, we have
\begin{align}\label{lemm_a}
\mathbf{x}_{ji}^\alpha \nabla_iW_{ij}^\beta=-\frac{\sigma_d}{h^{d+1}}\partial_R\hat{W}  \frac{\mathbf{x}_{ji}^{\alpha+\beta}}{\|\mathbf{x}_{ij}\|}.
\end{align}
Noting the fact $|\alpha+\beta|\geq 4$ and $d\leq 3$, there necessarily exists a multi-index $\bar{\beta}$ with $|\bar{\beta}|=1$ such that $2\bar{\beta}\leq \alpha+\beta.$ Consequently, it follows that
\begin{align}\label{lemm_b}
\mathbf{x}_{ji}^{\alpha+\beta}= \mathbf{x}_{ji}^{\alpha+\beta-2\bar{\beta}}\mathbf{x}_{ji}^{\bar{\beta}}\mathbf{x}_{ji}^{\bar{\beta}}.
\end{align}
Substituting \eqref{lemm_b} into \eqref{lemm_a} gives
\begin{align}
\mathbf{x}_{ji}^\alpha \nabla_iW_{ij}^\beta=\mathbf{x}_{ji}^{\alpha+\beta-2\bar{\beta}}\mathbf{x}_{ji}^{\bar{\beta}}\frac{\sigma_d}{h^{d+1}}\partial_R\hat{W}  \frac{\mathbf{x}_{ij}^{\bar{\beta}}}{\|\mathbf{x}_{ij}\|}=\mathbf{x}_{ji}^{\alpha+\beta-2\bar{\beta}}\mathbf{x}_{ji}^{\bar{\beta}}\nabla_iW_{ij}^{\bar{\beta}}.
\end{align}
Therefore, equation \eqref{eq1_bar} is verified. Moreover, if the function $W(\mathbf{x},h)$ adheres to the decay condition, we have $\partial_R\hat{W}\leq 0$, which implies $\mathbf{x}_{ji}^{\bar{\beta}}\nabla_iW_{ij}^{\bar{\beta}}\geq 0$.
By integrating this with \eqref{eq1_bar}, we can directly derive \eqref{eq2_bar}.
With this, the proof is concluded.
\end{proof}

We now present the specific regularity condition under which second-order accuracy respect with $h$ can be achieved in gradient particle approximation.
\begin{theorem}[Gradient approximation]\label{th3.1}
Assume that $f \in C^3(\Omega )$ and $W(\mathbf{x}, h)$ satisfies the decay condition. Denote
\begin{equation}\label{gamma}
 \gamma_{\alpha,\beta}=\left\{
\begin{array}{ll}
 1, & \hbox{if $\alpha =\beta$;} \\
 0, & \hbox{if $|\alpha|=1,2$ and $\alpha \neq\beta$,}
\end{array}
\right.
\end{equation}
where $\alpha, \beta$ are multi-indexes and $|\beta|=1$.
If
\begin{align}
|\sum_j v_j   \mathbf{x}_{ji}^\alpha \nabla_iW_{ij}^\beta-\gamma_{\alpha,\beta}|\leq h^2 \text{ and }  v_j \geq 0, \label{cond1-1}
\end{align}
then
\begin{align}\label{approx_grad}
\|\sum_j v_j f_{ji}\nabla_iW_{ij}-\nabla f(\mathbf{x}_i)\|_\infty \leq Ch^2\|f\|_{C^3(\Omega)}.
\end{align}
\end{theorem}
\begin{proof}
Applying the Taylor expansion for $f(\mathbf{x})$ at $\mathbf{x}_i$ gives
\begin{align}\label{Tylor}
f(\mathbf{x}_j)=f(\mathbf{x}_i)+ \sum_{|\alpha|=1} D^\alpha f_i \mathbf{x}_{ji}^\alpha+\frac{1}{2!}\sum_{|\alpha|=2} D^\alpha f_i \mathbf{x}_{ji}^\alpha+\frac{1}{3!}\sum_{|\alpha|=3} D^\alpha f(\xi_j) \mathbf{x}_{ji}^\alpha,
\end{align}
where $\xi_j$ is a point between $\mathbf{x}_i$ and $\mathbf{x}_j$.
Multiplying $v_j\nabla_iW_{ij}^\beta$ on both sides of \eqref{Tylor} and summing over $j$, we have
\begin{align}\label{eeqq1}
\sum_j v_j f_j  \nabla_iW_{ij}^\beta
=&f_i\sum_j v_j  \nabla_iW_{ij}^\beta +\sum_{|\alpha|=1}D^\alpha f_i\sum_j v_j  \mathbf{x}_{ji}^\alpha  \nabla_iW_{ij}^\beta \nonumber\\
&+\frac{1}{2!}\sum_{|\alpha|=2}D^\alpha f_i\sum_j v_j  \mathbf{x}_{ji}^\alpha  \nabla_iW_{ij}^\beta +\frac{1}{3!}\sum_{|\alpha|=3}\sum_j D^\alpha f(\xi_j) v_j   \mathbf{x}_{ji}^\alpha  \nabla_iW_{ij}^\beta.
\end{align}
Based on the definition of $\gamma_{\alpha,\beta}$, it follows that
\begin{align}\label{eeqq2-1}
\sum_{|\alpha|=1}D^\alpha f_i\sum_j v_j  \mathbf{x}_{ji}^\alpha  \nabla_iW_{ij}^\beta- D^\beta f_i
=\sum_{|\alpha|=1} D^\alpha f_i(\sum_j v_j  \mathbf{x}_{ji}^\alpha  \nabla_iW_{ij}^\beta-\gamma_{\alpha,\beta})
\end{align}
and
\begin{align}\label{eeqq2-2}
\frac{1}{2!}\sum_{|\alpha|=2}D^\alpha f_i\sum_j v_j  \mathbf{x}_{ji}^\alpha  \nabla_iW_{ij}^\beta
=\frac{1}{2!}\sum_{|\alpha|=2}D^\alpha f_i(\sum_j v_j  \mathbf{x}_{ji}^\alpha  \nabla_iW_{ij}^\beta-\gamma_{\alpha,\beta}).
\end{align}
Substituting \eqref{eeqq2-1} and \eqref{eeqq2-2} into equation \eqref{eeqq1} yields
\begin{align}\label{eeqq1_new}
&\sum_j v_j f_{ji}  \nabla_iW_{ij}^\beta -D^\beta f_i=\sum_{|\alpha|=1} D^\alpha f_i(\sum_j v_j  \mathbf{x}_{ji}^\alpha  \nabla_iW_{ij}^\beta-\gamma_{\alpha,\beta})\nonumber\\
&\quad +\frac{1}{2!}\sum_{|\alpha|=2}D^\alpha f_i(\sum_j v_j  \mathbf{x}_{ji}^\alpha  \nabla_iW_{ij}^\beta-\gamma_{\alpha,\beta})+\frac{1}{3!}\sum_{|\alpha|=3}\sum_j D^\alpha f(\xi_j) v_j   \mathbf{x}_{ji}^\alpha  \nabla_iW_{ij}^\beta.
\end{align}
Using Lemma \ref{barbeta}, we can estimate the last term on the right-hand side of \eqref{eeqq1_new} in the following manner:
\begin{align}\label{eeqq2-3}
&|\frac{1}{3!}\sum_{|\alpha|=3}\sum_j D^\alpha f(\xi_j) v_j  \mathbf{x}_{ji}^\alpha  \nabla_iW_{ij}^\beta|
\nonumber \\
&\leq \frac{1}{3!}\sum_{|\alpha|=3}\sum_j v_j|D^\alpha f(\xi_j)|  |\mathbf{x}_{ji}^{\alpha+\beta-2\bar{\beta}}| \mathbf{x}_{ji}^{\bar{\beta}} \nabla_iW_{ij}^{\bar{\beta}}\\
&\leq \frac{1}{3!}\sum_{|\alpha|=3}\max_j\{|D^\alpha f(\xi_j)||\mathbf{x}_{ji}^{\alpha+\beta-2\bar{\beta}}|\}\Big(1+(\sum_j v_j  \mathbf{x}_{ji}^{\bar{\beta}} \nabla_iW_{ij}^{\bar{\beta}}-1)\Big).\nonumber
\end{align}
Taking absolute values on both sides of \eqref{eeqq1_new} and applying \eqref{cond1-1}, \eqref{eeqq2-3}, we can obtain
\begin{align}\label{eeqq3}
|\sum_j v_j f_{ji}  \nabla_iW_{ij}^\beta -D^\beta f_i|\leq  Ch^2\|f\|_{C^3(\Omega)}.
\end{align}
Finally, \eqref{eeqq3} leads directly to the derivation of \eqref{approx_grad}. With this, the proof is concluded.
\end{proof}

Being a higher-order derivative, the particle approximation of the Laplace operator inherently demands more stringent conditions than the gradient approximation. In contrast to the conditions outlined in Theorem \ref{th3.1}, further requirements are imposed for the scenario where $|\alpha| = 0$, as presented in the subsequent theorem.
\begin{theorem}[Laplace approximation]\label{th3.2}
Assume that $f \in C^4(\Omega )$ and $W(\mathbf{x}, h)$ satisfies the decay condition. Denote
$$
 \bar{\gamma}_{\alpha,\beta}=\left\{
\begin{array}{ll}
 \gamma_{\alpha,\beta}, & \hbox{$|\alpha|=1,2$;} \\
 0, & \hbox{ $|\alpha|=0$,}
\end{array}
\right. $$
where $\gamma_{\alpha,\beta}$ is given by \eqref{gamma}.
If
\begin{align}
|\sum_j v_j   \mathbf{x}_{ji}^\alpha \nabla_iW_{ij}^\beta-\bar{\gamma}_{\alpha,\beta}|\leq h^2 \text{ and } v_j\geq 0,\label{cond2-11}
\end{align}
then
\begin{align}\label{approx_grad2}
\|2\sum_j v_j f_{ij}\frac{\mathbf{x}_{ij}\cdot \nabla_iW_{ij}}{\|\mathbf{x}_{ij}\|^2}-\Delta f(\mathbf{x}_i))\|_\infty\leq C h^2\|f\|_{C^4(\Omega)}.
\end{align}
\end{theorem}
\begin{proof}
Denote $Y_{ij}:=-2\frac{\mathbf{x}_{ij}\cdot \nabla_iW_{ij}}{\|\mathbf{x}_{ij}\|^2}$.
By using Lemma \ref{lem1-2}, it follows that
\begin{align}\label{co-linear}
 \mathbf{x}_{ij}Y_{ij}=-2\nabla_iW_{ij} \text{ and } \nabla_iW_{ij}^\beta=-\frac{1}{2}\mathbf{x}_{ij}^\beta Y_{ij}.
\end{align}
From \eqref{co-linear}, \eqref{cond2-11} can be rewritten as
\begin{align} \label{cond2-2}
|\frac{1}{2}\sum_j v_j   \mathbf{x}_{ji}^\alpha\mathbf{x}_{ji}^\beta  Y_{ij}-\bar{\gamma}_{\alpha,\beta}|\leq h^2.
\end{align}
Applying the Taylor expansion for $f$ at $\mathbf{x}_i$ gives
\begin{align}\label{Tylor2}
f(\mathbf{x}_j)=&f(\mathbf{x}_i)+ \sum_{|\alpha|=1} D^\alpha f_i \mathbf{x}_{ji}^\alpha+\frac{1}{2!}\sum_{|\alpha|=2} D^\alpha f_i \mathbf{x}_{ji}^\alpha\nonumber\\
&+\frac{1}{3!}\sum_{|\alpha|=3} D^\alpha f_i \mathbf{x}_{ji}^\alpha+\frac{1}{4!}\sum_{|\alpha|=4} D^\alpha f(\xi_j) \mathbf{x}_{ji}^\alpha,
\end{align}
where $\xi_j$ is a point between $\mathbf{x}_i$ and $\mathbf{x}_j$.
Multiplying $v_j Y_{ij}$ on both sides of \eqref{Tylor2} and summing over $j$, we have
\begin{align}\label{d}
\sum_j v_j f_{ji} Y_{ij} =& \sum_{|\bar{\alpha}|=1}D^{\bar{\alpha}} f_i\sum_j v_j  \mathbf{x}_{ji}^{\bar{\alpha}} Y_{ij} +\frac{1}{2!}\sum_{|\bar{\alpha}|=2}D^{\bar{\alpha}} f_i\sum_j v_j  \mathbf{x}_{ji}^{\bar{\alpha}} Y_{ij} \nonumber\\
&+ \frac{1}{3!}\sum_{|\bar{\alpha}|=3}D^{\bar{\alpha}} f_i\sum_j v_j  \mathbf{x}_{ji}^{\bar{\alpha}} Y_{ij} + \frac{1}{4!}\sum_{|\bar{\alpha}|=4}\sum_j v_j D^{\bar{\alpha}} f(\xi_j) \mathbf{x}_{ji}^{\bar{\alpha}} Y_{ij}.
\end{align}
By partitioning the multi-index $\bar{\alpha}$ into two distinct multi-indices $\alpha$ and $\beta$ where $|\beta| = 1$, denoted as $\bar{\alpha} = \alpha + \beta$, equation \eqref{d} can be reformulated as
\begin{align}\label{ddd}
\sum_j v_j f_{ji} Y_{ij} =& \sum_{|\alpha|=0,  |\beta|=1} D^\beta f_i\sum_j v_j  \mathbf{x}_{ji}^\beta Y_{ij} +\frac{1}{2!}\sum_{|\alpha|=1, |\beta|=1}D^{\alpha+\beta} f_i\sum_j v_j  \mathbf{x}_{ji}^\alpha \mathbf{x}_{ji}^\beta Y_{ij} \nonumber\\
&+ \frac{1}{3!}\sum_{|\alpha|=2, |\beta|=1}D^{\alpha+\beta} f_i\sum_j v_j  \mathbf{x}_{ji}^\alpha\mathbf{x}_{ji}^\beta Y_{ij}\\
& + \frac{1}{4!}\sum_{|\alpha|=3, |\beta|=1}\sum_j v_j D^{\alpha+\beta} f(\xi_j) \mathbf{x}_{ji}^\alpha\mathbf{x}_{ji}^\beta Y_{ij}.\nonumber
\end{align}
Based on the definition of $\bar{\gamma}_{\alpha,\beta}$, we have
\begin{align}\label{laplace_e1}
&\frac{1}{2!}\sum_{|\alpha|=1;|\beta|=1}D^{\alpha+\beta} f_i\sum_j v_j  \mathbf{x}_{ji}^\alpha \mathbf{x}_{ji}^\beta Y_{ij}
-\sum_{\alpha=\beta,|\beta|=1}D^{\alpha+\beta} f_i\nonumber \\
&=\sum_{|\alpha|=1;|\beta|=1}D^{\alpha+\beta} f_i
(\frac{1}{2!}\sum_j v_j  \mathbf{x}_{ji}^\alpha \mathbf{x}_{ji}^\beta Y_{ij}-\bar{\gamma}_{\alpha,\beta}),
\end{align}
and
\begin{align}\label{laplace_e2}
\sum_{\substack{|\alpha|=0,2;\\|\beta|=1}}D^{\alpha+\beta} f_i\sum_j v_j  \mathbf{x}_{ji}^\alpha\mathbf{x}_{ji}^\beta Y_{ij}=\sum_{\substack{|\alpha|=0,2;\\|\beta|=1}}D^{\alpha+\beta} f_i(\sum_j v_j  \mathbf{x}_{ji}^\alpha\mathbf{x}_{ji}^\beta Y_{ij}-\bar{\gamma}_{\alpha,\beta}).
\end{align}
Substituting \eqref{laplace_e1} and \eqref{laplace_e2} into equation \eqref{ddd}, then computing absolute values and applying \eqref{cond2-2}, we arrive at
\begin{align}\label{eeqq5}
|\sum_j v_j f_{ji} Y_{ij} - \sum_{\beta}D^{2\beta} f_i| \leq |\frac{1}{4!}\sum_{|\alpha|=3;\beta}\sum_j v_j D^{\alpha+\beta} f(\xi_j) \mathbf{x}_{ji}^\alpha\mathbf{x}_{ji}^\beta Y_{ij}|+Ch^2\|f\|_{C^3(\Omega)}.
\end{align}
It follows from \eqref{co-linear} and Lemma \ref{barbeta} that
\begin{align}\label{Laplace_b}
\big|\mathbf{x}_{ji}^\alpha\mathbf{x}_{ji}^\beta Y_{ij}\big|=\big|-2\mathbf{x}_{ji}^\alpha\nabla_iW_{ij}^\beta\big|
=\big|-2\mathbf{x}_{ji}^{\alpha+\beta-2\bar{\beta}}\big|\mathbf{x}_{ji}^{\bar{\beta}}\nabla_iW_{ij}^{\bar{\beta}}, \quad \forall |\alpha|=3.
\end{align}
Therefore, the first term on the right hand side of \eqref{eeqq5} can be estimated as follows
\begin{align}\label{eeqq6}
&\big|\frac{1}{4!}\sum_{|\alpha|=3;\beta}\sum_j v_j D^{\alpha+\beta} f(\xi_j) \mathbf{x}_{ji}^\alpha\mathbf{x}_{ji}^\beta Y_{ij}\big|\nonumber\\
&\leq \frac{1}{4!}\sum_{|\alpha|=3;\beta}\sum_j v_j \big|D^{\alpha+\beta} f(\xi_j)\big|
\big|2\mathbf{x}_{ji}^{\alpha+\beta-2\bar{\beta}}\big|\mathbf{x}_{ji}^{\bar{\beta}}\nabla_iW_{ij}^{\bar{\beta}} \nonumber\\
 &\leq \frac{1}{4!}\sum_{|\alpha|=3;\beta}\max_j\{|D^{\alpha+\beta} f(\xi_j)|
\big|2\mathbf{x}_{ji}^{\alpha+\beta-2\bar{\beta}}\big|\}\sum_j v_j \mathbf{x}_{ji}^{\bar{\beta}}\nabla_iW_{ij}^{\bar{\beta}}\\
&\leq Ch^2\|f\|_{C^4(\Omega)},\nonumber
\end{align}
where the last inequality follows from \eqref{cond2-11}.
Combining \eqref{eeqq5} and \eqref{eeqq6},  we obtain
\begin{align}\label{laplace_a}
|\sum_j v_j f_{ji} Y_{ij} -  \Delta f_i| \leq Ch^2\|f\|_{C^4(\Omega)}.
\end{align}
Equation \eqref{approx_grad2} can be directly deduced from \eqref{laplace_a}.
Thus, the proof concludes.
\end{proof}

The analysis of the truncation error outlined above can also be extended to other conventional operators like the Morris operator. Specifically, we can furnish the subsequent approximation error estimate for the Morris operator.
\begin{theorem}[Morris operator approximation]\label{th3.3}
Under the conditions of Theorem \ref{th3.2} and considering the assumption that the scalar function $a \in C^3(\Omega)$, the following holds
\begin{align}\label{approx_grad3}
\left\|\sum_j v_j f_{ij} \frac{(a_i+a_j)\mathbf{x}_{ij}\cdot \nabla_iW_{ij}}{\|\mathbf{x}_{ij}\|^2}-\mathcal{L} f(\mathbf{x}_i)\right\|_\infty\leq Ch^2\|a\|_{C^3(\Omega)}\|f\|_{C^4(\Omega)}.
\end{align}
\end{theorem}
\begin{proof}
Denote $\hat{R}=(a_i+a_j)-(2a_i+\nabla a_i\cdot \mathbf{x}_{ji}).$ Then the error in \eqref{approx_grad3} can be reformulated as follows
\begin{equation}\label{EE}
\sum_j v_j f_{ij} \frac{(a_i+a_j)\mathbf{x}_{ij}\cdot \nabla_iW_{ij}}{\|\mathbf{x}_{ij}\|^2}-\mathcal{L} f(\mathbf{x}_i)
=E_1+E_2+E_3,
\end{equation}
where $$E_1=a_i(\sum_j v_j f_{ji}Y_{ij}-\Delta f_i),\ E_2=\nabla a_i\cdot (\sum_j v_j f_{ji}\nabla_iW_{ij}-\nabla f_i),\
E_3=\sum_j v_j f_{ji}\hat{R}\frac{Y_{ij}}{2}.$$
From Theorem \ref{th3.2}, we can obtain
\begin{align}\label{E1}
|E_1|\leq Ch^2\|a\|_{C^0(\Omega)}\|f\|_{C^4(\Omega)}.
\end{align}
And it follows from Theorem \ref{th3.1} that
\begin{align}\label{E2}
  |E_2| \leq Ch^2\|a\|_{C^1(\Omega)}\|f\|_{C^3(\Omega)}.
\end{align}
Applying the Taylor expansion for $a(\mathbf{x})$ at $\mathbf{x}_i$ gives
\begin{align}
  \hat{R}=\frac{1}{2!}\sum_{|\alpha|=2}D^\alpha a_i \mathbf{x}_{ji}^\alpha
+\frac{1}{3!}\sum_{|\alpha|=3} D^\alpha a(\xi_j) \mathbf{x}_{ji}^\alpha:=R_1+R_2.
\end{align}
Similarly, there holds
\begin{align*}
f_{ji}= \sum_{|\alpha|=1}D^\alpha f_i \mathbf{x}_{ji}^\alpha +\frac{1}{2!}\sum_{|\alpha|=2}D^\alpha f(\xi_j) \mathbf{x}_{ji}^\alpha:=F_1+F_2.
\end{align*}
Thus, we can written $E_3$ as follows
\begin{align}\label{E33}
E_3&=\sum_j v_j F_1R_1\frac{Y_{ij}}{2}+\sum_j v_j F_1R_2\frac{Y_{ij}}{2}+\sum_j v_j F_2R_1\frac{Y_{ij}}{2}+\sum_j v_j F_2R_2\frac{Y_{ij}}{2}\nonumber\\
&:=A_1+A_2+A_3+A_4.
\end{align}
Now we use  inequality \eqref{cond2-2} to estimate $A_1$ as follows:
\begin{align}\label{A1}
|A_1|=|\frac{1}{2!}\sum_{|\alpha|=2,|\beta|=1}D^\alpha a_iD^\beta f_i\sum_j v_j \mathbf{x}_{ji}^\alpha\mathbf{x}_{ji}^\beta\frac{Y_{ij}}{2}|\leq Ch^2\|a\|_{C^2(\Omega)}\|f\|_{C^1(\Omega)}.
\end{align}
By using \eqref{Laplace_b} and \eqref{cond2-11}, we have
\begin{align}\label{A2}
|A_2|&=|\frac{1}{3!}\sum_{|\alpha|=3,|\beta|=1}D^\beta f_i\sum_j v_j D^\alpha a_i(\xi_j)  \mathbf{x}_{ji}^\alpha\mathbf{x}_{ji}^\beta\frac{Y_{ij}}{2}|\nonumber \\
&\leq\frac{1}{3!}\sum_{|\alpha|=3,|\beta|=1}|D^\beta f_i|\max_j\{|D^\alpha a_i(\xi_j)|\big|\mathbf{x}_{ji}^{\alpha+\beta-2\bar{\beta}}\big|\}\sum_j  v_j \mathbf{x}_{ji}^{\bar{\beta}}\nabla_iW_{ij}^{\bar{\beta}}\\
&\leq Ch^2\|a\|_{C^3(\Omega)}\|f\|_{C^1(\Omega)}.\nonumber
\end{align}
Similarly, we have
\begin{align}\label{A3}
|A_3| \leq Ch^2\|a\|_{C^2(\Omega)}\|f\|_{C^2(\Omega)}, \quad \mbox{and}\quad
|A_4| \leq Ch^3\|a\|_{C^3(\Omega)}\|f\|_{C^2(\Omega)}.
\end{align}
Combining the estimates of all $A_i,\  i=1,2,3,4$, we have
\begin{align}\label{E3}
|E_3|\leq Ch^2\|a\|_{C^3(\Omega)}\|f\|_{C^2(\Omega)}.
\end{align}
Substituting \eqref{E1}-\eqref{E2} and \eqref{E3} into  \eqref{EE}, we can derive the estimation given in \eqref{approx_grad3}. This concludes the proof.
\end{proof}

By applying the aforementioned analysis techniques to the truncation errors in function value particle approximations, we derive the following corollary. 
\begin{corollary}[Function approximation]\label{corollary}
Assume that $f \in C^2(\Omega )$.
If
\begin{align}\label{approx_function_cond}
|\sum_j v_j W_{ij}-1|\leq C h^2, |\sum_j v_j   \mathbf{x}_{ji}^\beta W_{ij}|\leq C h^2 (\forall |\beta|=1)  \text{ and }   v_j \geq 0,
\end{align}
then
\begin{align}\label{approx_function}
\|\sum_j v_j f_{j} W_{ij}- f(\mathbf{x}_i)\|_\infty \leq Ch^2\|f\|_{C^2(\Omega)}.
\end{align}
\end{corollary}
\begin{proof}
Applying the Taylor expansion for $f(\mathbf{x})$ at $\mathbf{x}_i$ gives
\begin{align}\label{appTylor}
f(\mathbf{x}_j)=f(\mathbf{x}_i)+ \sum_{|\alpha|=1} D^\alpha f_i \mathbf{x}_{ji}^\alpha +\frac{1}{2!}\sum_{|\alpha|=2} D^\alpha f(\xi_j) \mathbf{x}_{ji}^\alpha,
\end{align}
where $\xi_j$ is a point between $\mathbf{x}_i$ and $\mathbf{x}_j$.
Multiplying $v_j W_{ij}$ on both sides of \eqref{appTylor} and summing over $j$, it follows that
\begin{align}\label{appeeqq1}
\sum_j v_j f_j  W_{ij}
=&f_i\sum_j v_j  W_{ij} +\sum_{|\alpha|=1}D^\alpha f_i\sum_j v_j  \mathbf{x}_{ji}^\alpha  W_{ij} +\frac{1}{2!}\sum_{|\alpha|=2}\sum_j D^\alpha f(\xi_j) v_j   \mathbf{x}_{ji}^\alpha  W_{ij}.
\end{align}
Combining \eqref{appeeqq1} with $a$ and  equation \eqref{eeqq1} yields
\begin{align}\label{appeeqq1_new}
\sum_j v_j f_{j}  W_{ij} - f_i=&f_i(\sum_j v_j W_{ij}-1) +\sum_{|\alpha|=1}D^\alpha f_i\sum_j v_j  \mathbf{x}_{ji}^\alpha  \nabla_iW_{ij}^\beta \nonumber \\
&+\frac{1}{2!}\sum_{|\alpha|=2}\sum_j D^\alpha f(\xi_j) v_j   \mathbf{x}_{ji}^\alpha  W_{ij}.
\end{align}
By utilizing Lemma \ref{barbeta}, the last term on the right-hand side of \eqref{appeeqq1_new} can be estimated as follows
\begin{align}\label{appeeqq2-3}
|\frac{1}{2!}\sum_{|\alpha|=2}\sum_j D^\alpha f(\xi_j) v_j  \mathbf{x}_{ji}^\alpha  W_{ij}|
&\leq \frac{1}{2!}\sum_{|\alpha|=2}\sum_j v_j|D^\alpha f(\xi_j)|  |\mathbf{x}_{ji}^{\alpha}| W_{ij}\nonumber\\
&\leq \frac{1}{2!}\sum_{|\alpha|=2}\max_j\{|D^\alpha f(\xi_j)||\mathbf{x}_{ji}^{\alpha}|\}\sum_j v_j  W_{ij}.
\end{align}
Computing absolute values on both sides of \eqref{appeeqq1_new} and applying \eqref{approx_function_cond}, we can get
\begin{align}\label{appeeqq3}
|\sum_j v_j f_{j}W_{ij} -f_i|\leq  Ch^2\|f\|_{C^2(\Omega)}.
\end{align}
Finally, \eqref{appeeqq3} allows us to directly derive \eqref{approx_function}, which concludes the proof.
\end{proof}
\section{Volume reconstruction for SPH}\label{Sec4}
In traditional SPH method \cite{liu2010smoothed}, the  volume microelement $v_j$ in \eqref{laplacescheme} and \eqref{gradientscheme} is typically defined as the ratio of mass $m_j$ to density $\rho_j$, represented as  $v_j={m_j}/{\rho_j}$.
In the case of irregular particle distribution, the fulfillment of the conditions specified in Theorem \ref{th3.1} and Theorem \ref{th3.2} may no longer be guaranteed, requiring the development of novel approaches to tackle this challenge. In this section, we address this issue by reconstructing the volume microelement $v_j$ to satisfy the regularity conditions, followed by the presentation of the fully discrete scheme for the variable coefficient Poisson equation \eqref{eq1.1}.

Let $\Gamma_i = \{ j \mid W_{ij} > 0 \}$ denote the set of neighboring particle labels within the support of the kernel function centered at $\mathbf{x}i$.  The associated volumes of neighboring particle are $\{v_j\}_{j\in\Gamma_i}$. 
For each particle $i$, we reconstruct the particle volume $\{v_j\}_{j\in\Gamma_i}$ as $\{V_j\}_{j\in\Gamma_i}$ by minimizing the regularity condition. Specifically, we determine the  $\{V_j\}_{j\in\Gamma_i}$ by   minimize the optimization problems for the gradient and Laplacian operators, respectively:
\begin{align}
    \min_{ \{ V_j\}_{j\in \Gamma_i} } \sum_{\beta} \sum_{\alpha} \left| \sum_{j} V_j \mathbf{x}_{ji}^\alpha \nabla_i W_{ij}^\beta - \gamma_{\alpha,\beta} \right|,\quad \mbox{subject to}\quad V_j \geq 0,
    \label{min_equation}
\end{align}
\begin{align}
    \min_{ \{ V_j\}_{j\in \Gamma_i} } \sum_{\beta} \sum_{\alpha} \left| \sum_{j} V_j \mathbf{x}_{ji}^\alpha \nabla_i W_{ij}^\beta - \bar{\gamma}_{\alpha,\beta} \right|,\quad \mbox{subject to}\quad V_j \geq 0,
    \label{min_equation2}
\end{align}
where   \( \mathbf{x}_{ji}^\alpha \) denotes the \(\alpha\)-th component of the relative position vector \( \mathbf{x}_j - \mathbf{x}_i \),
 \( \nabla_i W_{ij}^\beta \) is the \(\beta\)-th component of the kernel gradient,
 \( \gamma_{\alpha,\beta} \) and \( \bar{\gamma}_{\alpha,\beta} \) are target consistency terms for the gradient and Laplacian approximations, respectively.
And the value of $V_j$ depends on $i$.
Taking the one-dimensional case as an example, equation \eqref{min_equation2} can be reformulated as the following non-negative linear least squares problem
\begin{align}\label{LS}
  \text{minimize } & \quad\|AV - b\|^2,\quad \mbox{subject to}\quad V \geq 0.
\end{align}
Here,
\begin{equation*}
A=\left[
  \begin{array}{ccc}
    \cdots &\nabla_iW_{ij} & \cdots  \\
   \cdots &\mathbf{x}_{ji} \nabla_iW_{ij}& \cdots  \\
\cdots &\mathbf{x}_{ji} \mathbf{x}_{ji} \nabla_iW_{ij}& \cdots  \\
  \end{array}
\right]
,  V=
\left[
  \begin{array}{c} 
    \vdots \\
      V_j\\
      \vdots
  \end{array}
\right], b=\left[
   \begin{array}{c}
0 \\
1 \\
0 \\
   \end{array}
 \right].
\end{equation*}
The column dimension of matrix $A$ is determined by $|\Gamma_i|$, where $|\Gamma_i|$ represents the count of neighboring particles in the interaction range of particle $i$.
The residual error in optimization problem \eqref{LS} should be reduced to  $O(h^4)$ to ensure the second-order convergence accuracy of VRSPH. In practice, for a given particle distribution with an appropriate choice of $h$, the residual error can reach machine precision (typically around $10^{-14}$). If an excessively small $h$ is selected for irregular particle distributions, the resulting the column rank of matrix $A$ is too small, violating the regularity condition and ultimately leading to order reduction in the overall solution.  Subsequently, we will introduce the concept of the covering radius to guide the setting of $h$.

We now establish notations of particle distribution and determine the value of $h$ based on the average particle radius and the covering radius.
For any $N\in\mathbb{N}$, define a particle distribution $X_N$ as
$X_N:=\{\mathbf{x}_i\in\bar{\Omega}; i=1,2,\cdots,N, x_i\neq x_j(i\neq j)\}.$
Denote $N_{in}$ as the total number of interior particles. Thus the interior particles can be expressed as $\{\mathbf{x}_i\in\bar{\Omega}/\partial \Omega; i=1,2,\cdots,N_{in} \}$ and the boundary particles are $\{\mathbf{x}_i\in \partial \Omega; i=N_{in}+1,N_{in}+2,\cdots,N \}.$
The covering radius of $X_N$ is defined as
\begin{align}\label{covering_radius}
\displaystyle r_N=\min \{r\in \mathbb{R}| \bigcup_i^N B(\mathbf{x}_i,r)  \supseteq \Omega\},
\end{align}
where $B(\mathbf{x}_i,r)$ denotes a ball centered at $\mathbf{x}_i$ with radius $r$,  with a similar definition noted in literature \cite{imoto2020truncation}.  
Define the average particle spacing of $X_N$ by $\Delta x = {|\Omega|^{\frac{1}{d}}}/{N^{\frac{1}{d}}},$ where $|\Omega|$ is the volume of the computational region. 
To better illustrate the relationship between covering radius and particle distribution irregularity, we present the schematic diagram at Figure \ref{fig:coverage-radius}. 
\begin{figure}[!htbp]
    \centering
    \begin{minipage}[t]{0.4\textwidth}
        \centering
        \includegraphics[width=\linewidth]{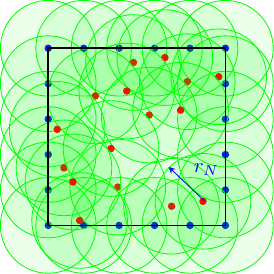} 
    \end{minipage} 
    \begin{minipage}[t]{0.33\textwidth}
        \centering
        \includegraphics[width=\linewidth]{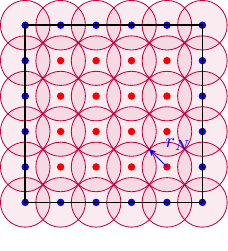} 
    \end{minipage}
    \hfill
    \begin{minipage}[t]{0.7\textwidth}
        \centering
        \includegraphics[width=\linewidth]{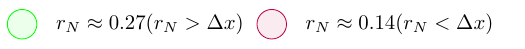} 
    \end{minipage}
    \caption{Schematic diagram of covering radius with $\Delta x = 0.2$ (left: random distribution; right: uniform distribution)}
    \label{fig:coverage-radius}
\end{figure}
By comparing covering radius for different particle distributions with identical particle numbers, we demonstrate that more irregular distributions correspond to larger covering radius. 
In SPH, $h$ is usually set to be $\kappa \Delta x$ where $\kappa$ is a constant depending on the smoothing function. In VRSPH, in order to assure that the minimum value of the optimization problems \eqref{min_equation} and \eqref{min_equation2} is $O(h^2)$, we set
\begin{align}\label{h}
h=\kappa\max\{\Delta x, r_N\}.
\end{align}
In a uniform particle distribution, a decrease in the average particle radius $\Delta x$ implies a reduction in interparticle spacing, thereby driving the truncation error asymptotically toward zero. 
Conversely, in a relatively disordered particle distribution, merely reducing the $\Delta x$ does not guarantee convergence of the truncation error. In such cases, the concept of covering radius $r_N$ (4.4) must be introduced to characterize the spatial properties of the particle arrangement. Crucially, it is the reduction of $r_N$—not just $\Delta x$—that ensures the truncation error approaches zero.
If the particles are uniformly distributed, $r_N$ is less than $\Delta x$. If the particle distribution is irregular, $r_N$ may be greater than $\Delta x$. 
In the latter scenario, an efficient estimation of \( r_N \) can be achieved via a lightweight computational procedure. In this work, we approximate \( r_N \) in terms of the maximum edge length \( l_{\text{max}} \) (commonly referred to as the mesh size) of a simplicial mesh constructed from the particle distribution. The approximation strategy consists of the following steps: first, a simplicial mesh is generated from the given set of particles using Delaunay triangulation. We then propose the estimate
\begin{align}
r_N \approx \beta \cdot l_{\text{max}},\label{Approxr_N}
\end{align}
where the factor \( \beta \) denotes the minimum ratio of the circumradius to the edge length over all simplices in the mesh. The values of \( \beta \) are given by \( \beta = 1/2 \) for a segment (1-simplex), \( \beta = 1/\sqrt{3} \) for a triangle (2-simplex), and \( \beta = \sqrt{6}/4 \) for a tetrahedron (3-simplex).
 
Once  $V_j$  is determined, the particle approximation in the SPH method can be immediately constructed. Specifically, by substituting $V_j$   into the conventional SPH approximation formulation in place of  $v_j$, a high-precision particle approximation is achieved. For instance, using  $V_j$, we can derive a function and its first and second derivatives approximation  as follows:
\begin{align} 
\text{Function} &&\langle f(\mathbf{x}_i)\rangle_{VRSPH}&=\sum_j V_jf_{j}  W_{ij} ,\\
\text{Gradient} &&\langle\nabla f(\mathbf{x}_i)\rangle_{VRSPH}&=-\sum_j V_jf_{ij} \nabla_iW_{ij} ,\\
 \text{Laplacian} &&\langle \Delta \mathbf{A}(\mathbf{x}_i)\rangle_{VRSPH}&= 2\sum_j V_j\mathbf{A}{ij}\frac{ \mathbf{x}{ij}\cdot \nabla_iW_{ij} }{|\mathbf{x}_{ij}|^2}. \label{LaplacianVRSPH}
\end{align} 
To enhance clarity, we summarize the Laplacian approximation of VRSPH method in Algorithm \ref{alg:VRSPH_Laplacian}.
\begin{algorithm}[ht] 
\caption{VRSPH method for Laplacian Approximation}
\label{alg:VRSPH_Laplacian}
\begin{algorithmic}[1]
\Require Particle positions $X_N$, field values $\{\mathbf{A}(\mathbf{x}_i)\}_{i=1}^N$
\Ensure Approximated Laplacian $\langle \Delta \mathbf{A}(\mathbf{x}_i)\rangle_{\text{VRSPH}}$ for all particles
\State Estimate covering radius $r_N$ by $\Delta x$ or using \eqref{Approxr_N}
\State Compute kernel influence radius $h$ using Equation~\eqref{h}
\For{each particle $i = 1, 2, \cdots, N$}
    \State Reconstruct neighbor volumes $\{V_j\}$ by solving optimization problem~\eqref{min_equation2}
    \State Compute Laplacian approximation $\langle \Delta \mathbf{A}(\mathbf{x}_i)\rangle_{\text{VRSPH}}$ using Equation~\eqref{LaplacianVRSPH} with reconstructed volumes
\EndFor
\State \Return $\langle \Delta \mathbf{A}(\mathbf{x}_i)\rangle_{\text{VRSPH}}$ for all particles
\end{algorithmic}
\end{algorithm}
The VRSPH approximation for the function and the first-order derivative is omitted here, as they are similar to the Laplacian approximation.

Now we can construct the following numerical scheme to solve \eqref{eq1.1}-\eqref{Dirichlet}: For $1\leq i\leq N$, given $a_i, f_i$, compute grid function $U=\{U_i\}_{i=1}^N$  by
\begin{align}\label{scheme}
\begin{array}{rl}
\mathcal{L}_hU_i:=\sum\limits_{1\leq j \leq N} V_jU_{ij}\frac{(a_i+a_j)\mathbf{x}_{ij}\cdot \nabla_iW_{ij}}{\|\mathbf{x}_{ij}\|^2}&= f_i, \quad 1\leq i \leq N_{in},\\
U_i&=0, \quad N_{in}+1\leq i \leq N,
\end{array}
\end{align}
where $U_{ij}=U_i-U_j$ and $V_j$ is determined by the optimization problem \eqref{min_equation2}. 
By using \eqref{scheme}, we can define  the discrete Laplacian matrix $L$ as follows.
\begin{align}\label{Ldefinition}
  L&=\left(
      \begin{array}{cccc}
        -\sum_{j\neq 1} a_{1,j} & a_{1,2} & \cdots & a_{1,N_{in}} \\
        a_{2,1} & -\sum_{j\neq 2} a_{2,j} &   & a_{2,N_{in}} \\
        \vdots &   & \ddots & \vdots \\
        a_{N_{in},1} & a_{N_{in},2} & \cdots &  -\sum_{j\neq N_{in}} a_{N_{in},j} \\
      \end{array}
    \right),
\end{align}
where 
\begin{align*}
    a_{i,j}= - V_j \frac{(a_i+a_j)\mathbf{x}_{ij}\cdot\nabla_iW_{ij}}{\|\mathbf{x}_{ij}\|^2}
    \quad \forall  i\neq j.
  \end{align*}
In fact, $-L$ being an M-matrix, which allows its inverse to be computed more efficiently and facilitates the construction of numerical schemes based on the maximum bound principle \cite{ju2021maximum}. 
\begin{remark}
 The volume microelements $V_j$, solved through the non-negative linear least squares problem \eqref{LS}, do not represent the actual volumes of particle $j$. Instead, they serve as weights in the integration of the kernel function for particle $i$. Many of the $V_j$ values are zero, contributing to matrix sparsity.
\end{remark}
\begin{remark}
  We emphasize that the second-order convergence of VRSPH is rigorously established with respect to the kernel radius, i.e., $O(h^2)$. 
The convergence behavior may be different with respect to $\Delta x$, depending on the relationship between the $r_N$ and $\Delta x$. 
In  general case where the particle distribution follows $r_N \leq (\Delta x)^p$ with $p \leq 1$, we obtain $h\leq C(\Delta x)^p$. Thus, we can achieve a convergence rate of $O((\Delta x)^{2p})$ with respect to $\Delta x$. 
Numerous particle distributions satisfy the condition $r_N\leq \sqrt{d}\Delta x$, including both uniform distributions and randomly perturbed particles. However, for completely random particle arrangements where $p<1$ typically holds, order reduction is considered acceptable for such highly irregular distributions. 
\end{remark}

\begin{remark}
The choice of \(h\) in Equation (4.5) is motivated by theoretical rigor and analytical convenience. In practice, this condition can be relaxed suitably. First, if prior information is available (e.g., for particle configurations that are randomly perturbed or uniformly distributed such that \(r_N \le \sqrt{d}\,\Delta x\)), computing \(r_N\) explicitly can be avoided entirely. Second, even in the absence of such prior information, the selection of \(h\) need not depend on the exact value of \(r_N\); a rough estimate of an upper bound for \(r_N\) can be used instead.
\end{remark}

\section{Error estimates  for the variable coefficient Poisson equation}\label{Sec5}
In this section, we present the error analysis of \eqref{scheme} for solving the variable coefficient Poisson equation \eqref{eq1.1}-\eqref{Dirichlet} with irregular particle distributions. We first rigorously demonstrate that with irregular particle distributions, the discrete maximum principle is preserved for the given fully discretized scheme \eqref{scheme}. Subsequently, we establish an error estimation in the infinity norm based on this principle.

\begin{theorem}[Discrete Maximum Principle] \label{DMP}
Assume that the smoothing kernel function $W(\mathbf{x},h)$ satisfies the decay condition.
Consider a grid function $U=\{U_i\}_{i=1}^N$ and $ a_i>0\ (1\leq i\leq N)$. If the discrete Laplace operator $\mathcal{L}_h$ satisfies
\begin{align}
\mathcal{L}_h U_i =\sum_{1\leq j \leq N} V_jU_{ij}\frac{(a_i+a_j)\mathbf{x}_{ij}\cdot \nabla_iW_{ij}}{\|\mathbf{x}_{ij}\|^2}\geq 0, \quad 1\leq i\leq N_{in},
\end{align}
then $U$ attains its maximum on the boundary. On the other hand, if $ \mathcal{L}_h U_i\leq 0$ for any  $1\leq i \leq N_{in} $,  then $U$ attains its minimum on the boundary.
\end{theorem}

\begin{proof}
We prove the theorem by contradiction. Suppose that the conclusion is
not true, so $U$ has its maximum at an interior grid point $\mathbf{x}_{n_0}$, i.e., $U_{n_0}=\max\limits_{1\leq i \leq N} U_i$ with $1\leq n_0\leq N_{in}$. Then there holds
\begin{equation}\label{M-eq1}
U_{n_0i}:=U_{n_0}-U_{i} \geq 0 , \quad 1\leq i \leq N.
\end{equation}
It follows from Lemma \ref{lem1-2} that
\begin{equation}\label{M-eq2}
x_{n_0j}\cdot\nabla_{n_0}W_{n_0j}\leq 0.
\end{equation}
In view of the fact $a_i>0$ and \eqref{M-eq2}, the inequality
$$ \mathcal{L}_h U_{n_0}=\sum_{1\leq j \leq N} V_jU_{n_0j}\frac{(a_{n_0}+a_j)\mathbf{x}_{n_0j}\cdot \nabla_{n_0}W_{n_0j}}{\|\mathbf{x}_{n_0j}\|^2} \geq 0$$
 contradicts with \eqref{M-eq1}  unless all $U_i$ at the neighbors of $n_0$ have the same value $U_{n_0}$. This implies that $U$ has its maximum at neighboring points of $x_{n_0}$, and the same argument can be applied finite times until the boundary is reached. Then we know that $U$ also has its maximum at boundary particles. We can conclude that if $U$ has its maximum any interior point, then $U$ is a constant.

We can use a similar argument to prove the case of $ \mathcal{L}_h U_i\leq 0$.  Then we complete the proof.
\end{proof}

\begin{corollary}[M-matrix]\label{M-matrix}
Assume that $W(\mathbf{x},h)$ satisfies the decay condition and $a_i>0$ ($1\leq i \leq N$). $L$ is the discrete Laplacian matrix given by \eqref{Ldefinition}.
It follows that $-L$ is a non-singular M-matrix.
\end{corollary}
\begin{proof}
Let $\bar{U}=[U_1,\cdots,U_{N_{in}}]^T$.
$(-L)_i$ denotes the $i$th row of the matrix $(-L)$. In view of the fact $U_i=0$ if $N_{in}+1\leq i \leq N$, there holds
$$(-L)_i\bar{U}=-\mathcal{L}_h U_i, 1\leq i\leq N_{in}.$$
If $(-L)\bar{U}\geq 0$, then $\mathcal{L}_h U_i\leq 0$ for any $1\leq i \leq N_{in}$. By using Theorem \ref{DMP}, we have $U$  attains its minimum on the boundary, i.e.,
$$\bar{U}\geq \min\limits_{1\leq i\leq N_{in}} U_i = \min_{N_{in}+1 \leq i\leq N} U_i =0.$$
 In other words, we can obtain $\bar{U}\geq 0$ from $(-L)\bar{U}\geq 0$, which means $-L$ is monotonic.
Since monotone real Z-matrices are known to be nonsingular M-matrices, it follows that $-L$ must also be an M-matrix. The proof is complete.
\end{proof}

To establish the error estimation, we need to have the following lemma.
\begin{lemma}\label{DMP2}
Assume that $\Omega=[0,l]^d$, $W(\mathbf{x},h)$ satisfies the decay condition and  $a(\mathbf{x})\geq a_{\min}> 0$.
Let $U$ be a grid function that satisfies
$$ \mathcal{L}_h U_i=\sum_{1\leq j \leq N} V_jU_{ij}\frac{(a_i+a_j) \mathbf{x}_{ij}\cdot \nabla_iW_{ij}}{\|\mathbf{x}_{ij}\|^2}=f_i, \quad 1\leq i\leq N_{in}$$
with an homogeneous Dirichlet boundary condition.
Under the conditions of Theorem \ref{th3.3} and assuming $h$ satisfies the following condition
\begin{align}\label{small_h}
Ch^2\|a\|_{C^3(\Omega)}\|f\|_{C^4(\Omega)}\leq 1,
\end{align}
we have
\begin{align}\label{conclu_DMP2}
\|U\|_\infty=\max\limits_{1\leq i\leq N}|U_i|\leq e^{\frac{\|a\|_{C^1(\Omega)}+2}{a_{\min}} l}  \|f\|_\infty,
\end{align}
where $C$ is a constant given in \eqref{approx_grad3}.
\end{lemma}
\begin{proof}
Define a continuous function $g(\mathbf{x})= e^{\lambda x}$ with $\lambda=\frac{\|a\|_{C^1(\Omega)}+2}{a_{\min}}$.
It follows that
\begin{align}\label{estimate_1}
\mathcal{L}g(\mathbf{x}) &=a \Delta g(\mathbf{x})+\nabla a\cdot \nabla g(\mathbf{x}) =a\lambda^2 e^{\lambda x}+\partial_xa(\mathbf{x})\lambda e^{\lambda x}\nonumber\\
&\geq \lambda e^{\lambda x}(a_{\min}\lambda-\|a\|_{C^1(\Omega)})
\geq 2.
\end{align}
Define an interpolation function $\{w_i\}_{i=1}^N$ associated with $g$ satisfying  $w_i=g(\mathbf{x}_i)$.
By using Theorem \ref{th3.3}, \eqref{small_h} and \eqref{estimate_1} , we have
\begin{align}\label{estimate_2}
\mathcal{L}_h w_i&\geq \mathcal{L} g(\mathbf{x}_i)-Ch^2\|a\|_{C^3(\Omega)}\|f\|_{C^4(\Omega)}\nonumber\\
&\geq 1+(1-Ch^2\|a\|_{C^3(\Omega)}\|f\|_{C^4(\Omega)})\geq 1, \quad \forall 1\leq i \leq N_{in}.
\end{align}
It follows from \eqref{estimate_2} that
$$\mathcal{L}_h (U_i-\|f \|_\infty w_i) = (\mathcal{L}_hU_{i} - \| f\|_\infty)+(1-\mathcal{L}_h w_i) \|f\|_\infty  \leq 0,$$
and
$$\mathcal{L}_h (U_i+\|f \|_\infty w_i) = (\mathcal{L}_hU_{i} + \|f\|_\infty)+(-1+\mathcal{L}_h w_i) \|f\|_\infty\geq 0.$$
Therefore, by using Theorem \ref{DMP}, $\{U_i+\|f \|_\infty w_i\}_{i=1}^N$ has its maximum on the boundary, while $\{U_i-\|f \|_\infty w_i\}_{i=1}^N$ has its minimum on the boundary. In other word, there hold that
$$ U_i-\|f \|_\infty w_i\geq \min\limits_{N_{in}+1\leq i \leq N} (U_i-\|f \|_\infty w_i), \forall 1\leq i \leq N$$
and
$$U_i+\|f \|_\infty w_i\leq \max_{N_{in}+1\leq i \leq N} (U_i+\|f \|_\infty w_i), \forall 1\leq i \leq N.$$
Noting the fact that $w_i\geq 0 (1\leq i\leq N)$ and $U_{i}$ is zero on the boundary, we have
\begin{align}\label{minU}
 U_i\geq U_i-\|f \|_\infty w_i\geq  -\|f \|_\infty \max_{N_{in}+1\leq i \leq N}\|w_i\|_\infty
\end{align}
and
\begin{align}\label{maxU}
U_i\leq  U_i+\|f \|_\infty w_i\leq \|f \|_\infty \max_{N_{in}+1\leq i \leq N}\|w_i\|_\infty.
\end{align}
By the definition of $\{w\}_{i=1}^N$, it follows that
\begin{align}\label{maxw}
  \max_{N_{in}+1\leq i \leq N}\|w_i\|_\infty\leq e^{\frac{\|a\|_{C^1(\Omega)}+2}{a_{\min}} l}.
\end{align}
From \eqref{minU}-\eqref{maxw}, we can obtain
$$ |U_{i}| \leq  e^{\frac{\|a\|_{C^1(\Omega)}+2}{a_{\min}} l} \| f \|_\infty, \forall 1\leq i \leq N.$$
Hence, we get \eqref{conclu_DMP2} and thus complete the proof.
\end{proof}

Through the construction of an auxiliary function $g(\mathbf{x})$, we have demonstrated that $\|U\|_\infty$ is bounded by $\|\mathcal{L}_hU\|_\infty$. Next, we will present the error estimation of the numerical scheme \eqref{scheme}.
\begin{theorem}\label{convergence}
Assume that $\Omega=[0,l]^d$, $u(\mathbf{x})\in C^4(\Omega)$ and $a(\mathbf{x})\in C^3(\Omega)$. Let $u(\mathbf{x})$ and $U$ be the solutions of \eqref{eq1.1} and \eqref{scheme}, respectively.
If regularity condition \eqref{cond2-11} holds and $h$ satisfying \eqref{small_h}, then the error $E_i:=U_i-u(\mathbf{x}_i)$ satisfies:
$$\|E\|_\infty=\max\limits_{1\leq i \leq N}|E_i| \leq Ch^2e^{\frac{\|a\|_{C^1(\Omega)}+2}{a_{\min}} l}\|a\|_{C^3(\Omega)}\|u\|_{C^4(\Omega)},$$
where $C$ is a constant given in \eqref{approx_grad3}.
\end{theorem}
\begin{proof}
Denote $T_i$ as the local truncation error at $\mathbf{x}_i$, i.e.,
\begin{align}\label{denoteT}
T_i:=\mathcal{L} u(\mathbf{x}_i)-\sum_{1\leq j\leq N} V_j \Big(u(\mathbf{x}_i)-u(\mathbf{x}_j)\Big) \frac{(a_i+a_j)\mathbf{x}_{ij}\cdot \nabla_iW_{ij}}{\|\mathbf{x}_{ij}\|^2},  \quad 1\leq i \leq N_{in}.
\end{align}
From Theorem \ref{th3.3}, we have $\|T_i\|_\infty\leq Ch^2\|a\|_{C^3(\Omega)}\|u\|_{C^4(\Omega)}$.
It follows from  \eqref{eq1.1} and \eqref{denoteT} that
\begin{equation}\label{Laplace_ui}
\sum_{1\leq j\leq N} V_j \Big(u(\mathbf{x}_i)-u(\mathbf{x}_j)\Big) \frac{(a_i+a_j)\mathbf{x}_{ij}\cdot \nabla_iW_{ij}}{\|\mathbf{x}_{ij}\|^2} = f_{i} - T_{i}, \quad 1\leq i\leq N_{in}.
\end{equation}
Subtracting \eqref{Laplace_ui} from \eqref{scheme}, we obtain
 $\mathcal{L}_h E_{i} = T_{i}.$
By {using} Lemma \ref{DMP2}, we get
$$\|E\|_\infty \leq  e^{\frac{\|a\|_{C^1(\Omega)}+2}{a_{\min}} l} \|T\|_\infty\leq Ch^2e^{\frac{\|a\|_{C^1(\Omega)}+2}{a_{\min}} l}\|a\|_{C^3(\Omega)}\|u\|_{C^4(\Omega)}.$$
The proof is completed.
\end{proof}

\section{Numerical experiments}\label{Sec6}
In this section, we focus on evaluating the particle approximation error of VRSPH for interior particles and compare it with established high-accuracy SPH methods such as FPM, CSPH, and KGF. Additionally, we assess the performance of the new method when the kernel function is truncated by boundaries with uniform particles. We also examine the convergence order of the VRSPH method for solving the variable coefficient Poisson equation. In our experiments, unless otherwise noted, we consider $\Omega=[-1,1]^d$ with $d=1,2.$ The test function $f(\mathbf{x})$ is given as follows:
\begin{align*}
  f(\mathbf{x})= \left\{
          \begin{array}{ll}
            \cos(\pi x), & \hbox{$d=1$;} \\
            \cos(\pi x)\cos(\pi y), & \hbox{$d=2$.}
          \end{array}
        \right.
\end{align*}
And we select the cubic spline kernel \eqref{cubic-spline} as the smoothing function.

\subsection{Accuracy Test for interior particles}
To accurately test the particle approximation accuracy of interior particles, we placed a sufficient number virtual particles outside the boundary to prevent the truncation of the kernel function due to the boundary (as shown in Figure \ref{irr2D}).
\begin{figure}[!h]
  \centering
\begin{tikzpicture}
\draw (-5,0) -- (5,0);
\foreach \x in {-5,-4,-3,-2,-1,0,1,2,3,4,5}
\draw (\x,0.2) -- (\x,0);
\draw[fill,blue] (-5.1,-0.1) rectangle (-4.9,0.1);
\node at (-5,-0.1)[below] {$-1$};
\draw[fill,blue] (5.1,-0.1) rectangle (4.9,0.1);
\node at (5,-0.1)[below] {$1$};
\node at (0,-0.1)[below] {$0$};
\fill[red] (-4.35,0) circle (2pt);
\fill[red] (-2.6054,0) circle (2pt);
\fill[red] (-2.1635,0) circle (2pt);
\fill[red] (-1.3414,0) circle (2pt);
\fill[red] (0.3389,0) circle (2pt);
\fill[red] (1.2473,0) circle (2pt);
\fill[red] (2.0773,0) circle (2pt);
\fill[red] (3.1965,0) circle (2pt);
\fill[red] (4.2965,0) circle (2pt);
\draw (-5.6,-0.1) rectangle (-5.4,0.1);
\draw (5.6,-0.1) rectangle (5.4,0.1);
\draw (-6.1,-0.1) rectangle (-5.9,0.1);
\draw (6.1,-0.1) rectangle (5.9,0.1);
\draw (-5.1,1.6) rectangle (-4.9,1.4);
\node at (-4.9,1.5) [right] {virtual particles};
\draw[fill,blue] (-5.1,1.1) rectangle (-4.9,0.9);
\node at (-4.9,1) [right] {boundary particles};
\fill[red] (-5.,0.5) circle (2pt);
\node at (-4.9,0.5) [right] {random interior particles};
\end{tikzpicture}
\includegraphics[width=0.3\linewidth]{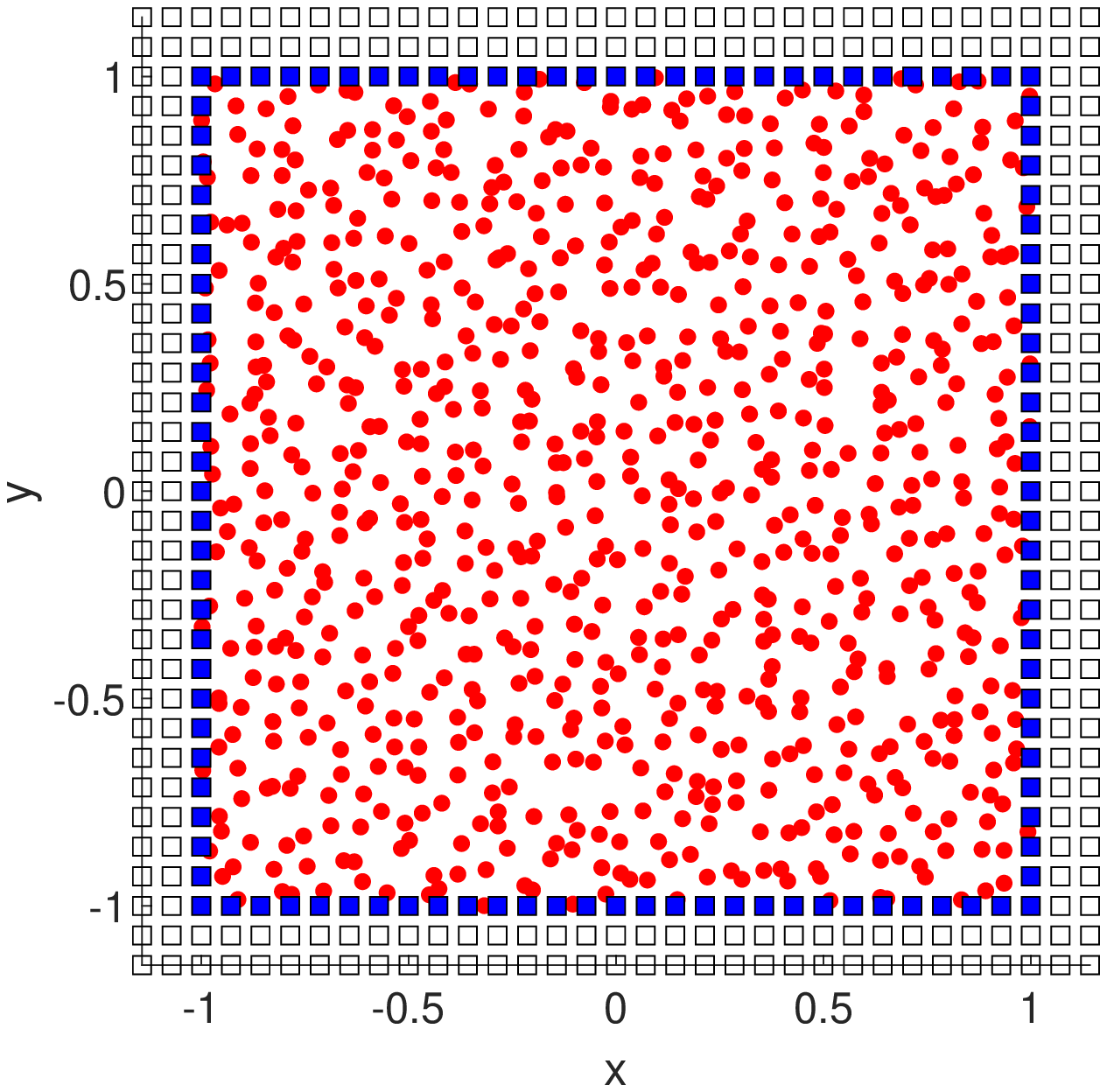}
\caption{Randomly perturbed particles in 1D (top) 2D (bottom)}\label{irr2D}
\end{figure}
\begin{example}
  \textbf{Uniform particles.}
We consider uniform particle distribution in a two-dimensional space. Let $h=3\Delta x$ and $N^{1/d}=2^k$, where $k$ ranges from 3 to 7. The particle approximation errors for the gradient and Laplace operators are being tested individually. The decay of truncation errors can be seen in Figure \ref{non-uniform}.
\end{example}
\begin{figure}[!h]
  \centering
  \includegraphics[width=0.4\linewidth]{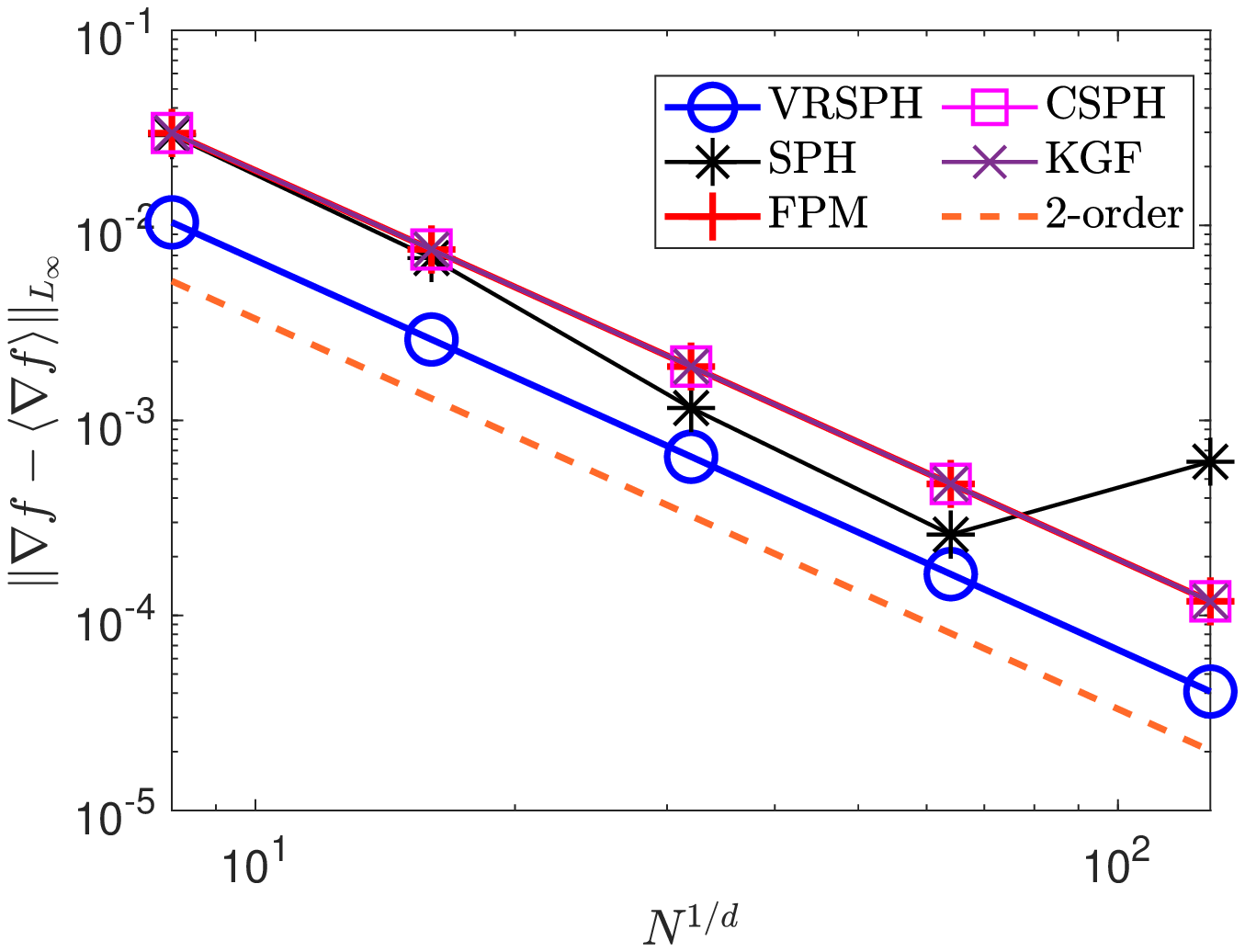}
   \includegraphics[width=0.4\linewidth]{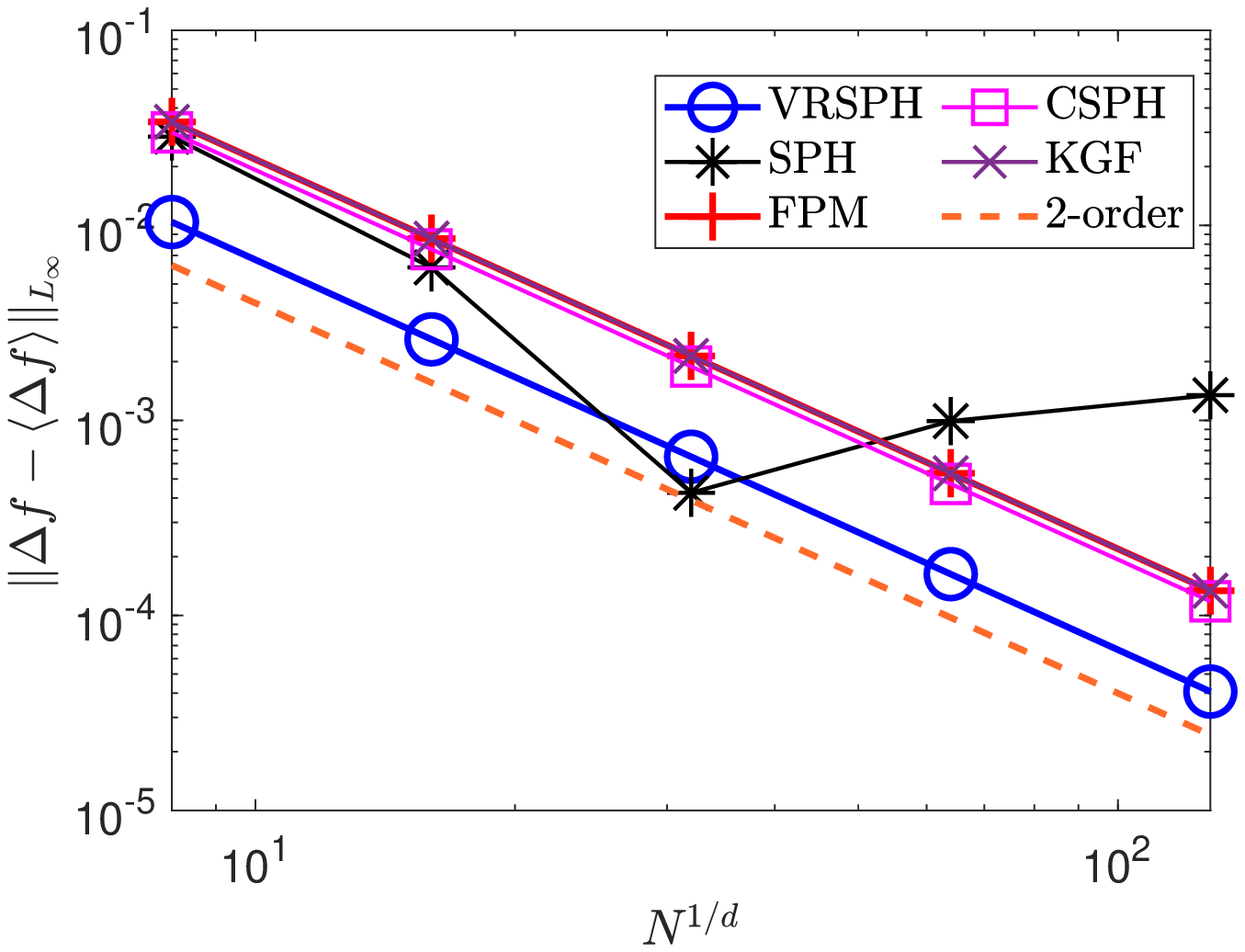}
  \caption{Gradient (left) and Laplace (right) approximation on uniform distribution in 2D}\label{non-uniform}
\end{figure}

From Figure \ref{non-uniform}, it can be seen that the VRSPH method achieves second-order accuracy in approximating the gradient and Laplace operators.
The truncation error of VRSPH is lower than that of the other four methods for the same $N^{1/d}$.
Additionally, the truncation error of SPH in 2D may increase as $h$ is decreased with constant $\Delta x/h$. As discussed in \cite{quinlan2006truncation},  first-order consistent methods such as CSPH have been shown to eliminate this divergent behavior.

Notice that the computation of the covering radius $r_N$ for arbitrary irregular particle distributions is challenging. For randomly perturbed particles, $r_N$ is smaller than $\Delta x$. To test convergence accuracy, we use randomly perturbed particles to simulate irregularly distributed particles, where $h$ varies proportionally with $N^{-1/d}$.

\begin{example}
\textbf{Randomly perturbed particles.}
We allow particles uniformly distributed on $\Omega$ and undergo random perturbations. The perturbation displacement in the direction of the coordinate axes following a random distribution $U(-0.5 \Delta x, 0.5 \Delta x)$ as shown in Figure \ref{irr2D}.  
The corresponding 3D particle distribution is  similar to the 2D case and is therefore omitted for brevity.   
We set $h=3\Delta x$ and test the particle approximation error of the gradient and Laplace operators. The truncation errors are presented in Figures \ref{gradient1D2D} and \ref{Laplace1D2D}.
\end{example}
\begin{figure}[!h]
\label{gradient1D2D}
  \centering
  \includegraphics[width=0.3\linewidth]{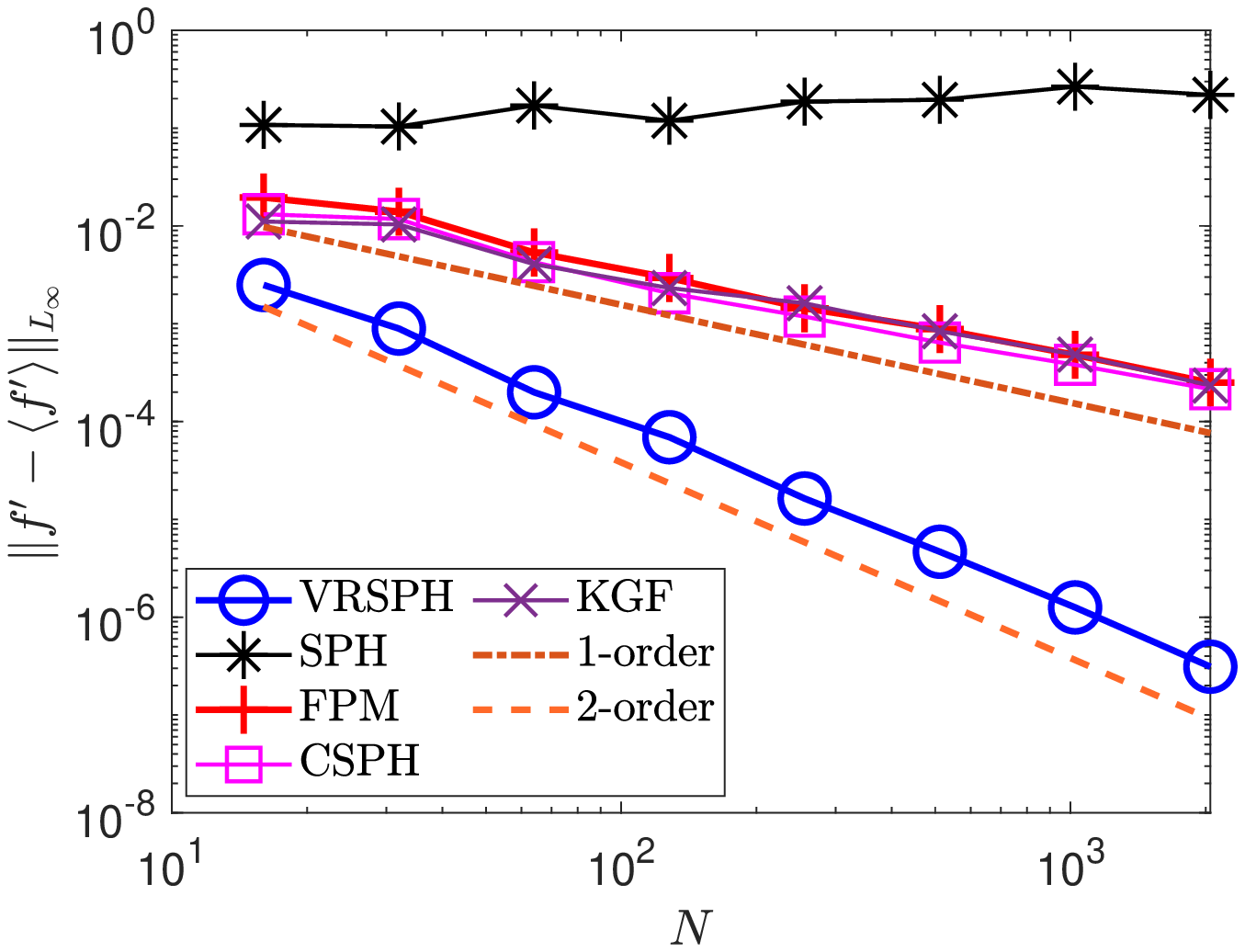}
  \includegraphics[width=0.3\linewidth]{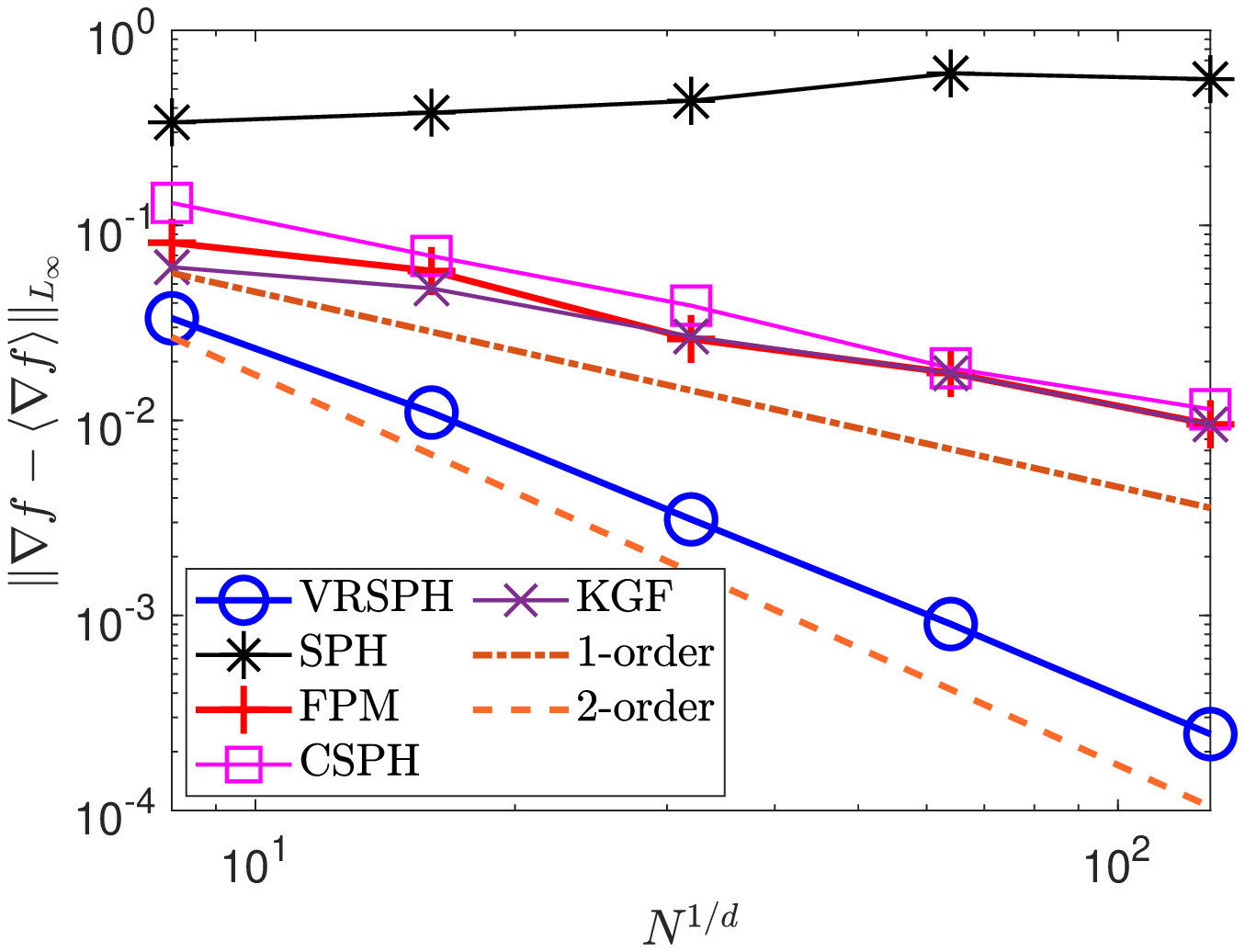}
  \includegraphics[width=0.3\linewidth]{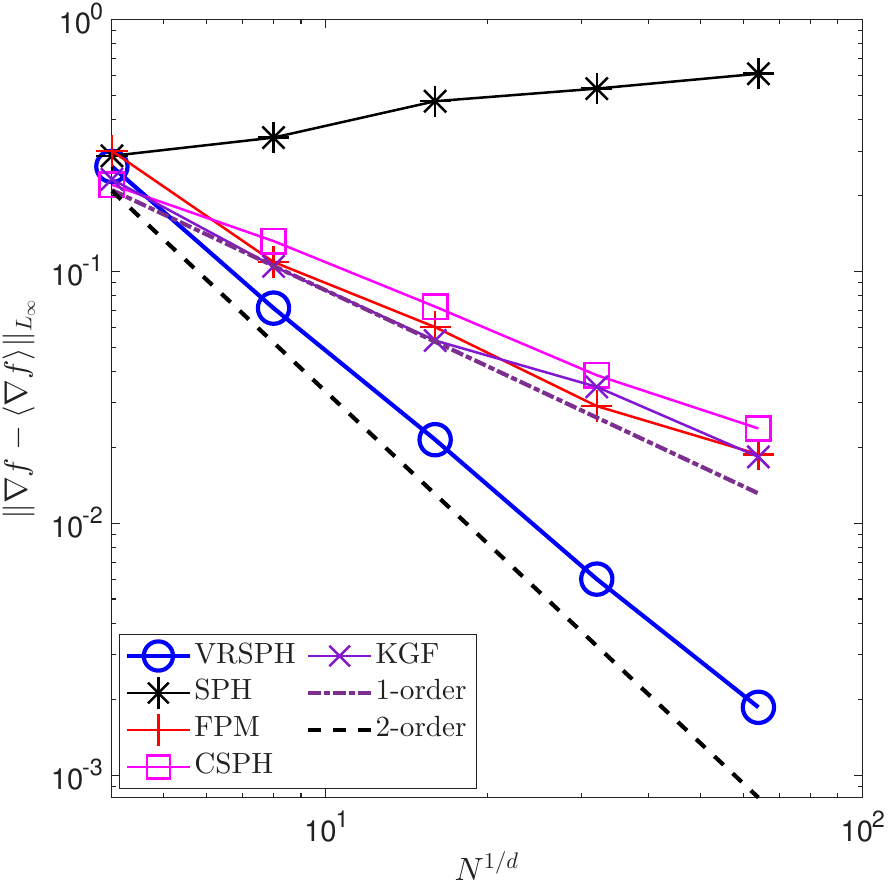}
  \caption{Truncation error of the gradient in 1D (left), 2D (middle) and 3D (right)}
\end{figure}

Based on Figure \ref{gradient1D2D}, it is evident that the traditional SPH method does not decrease the truncation error of the gradient approximation as the number of particles increases. Conversely, CSPH, FPM, and KGF methods show improvement in reducing the truncation error of the gradient to first-order accuracy. Our proposed VRSPH method enhances the gradient approximation error to second-order accuracy.
\begin{figure}[!h]
\label{Laplace1D2D}
  \centering
  \includegraphics[width=0.3\linewidth]{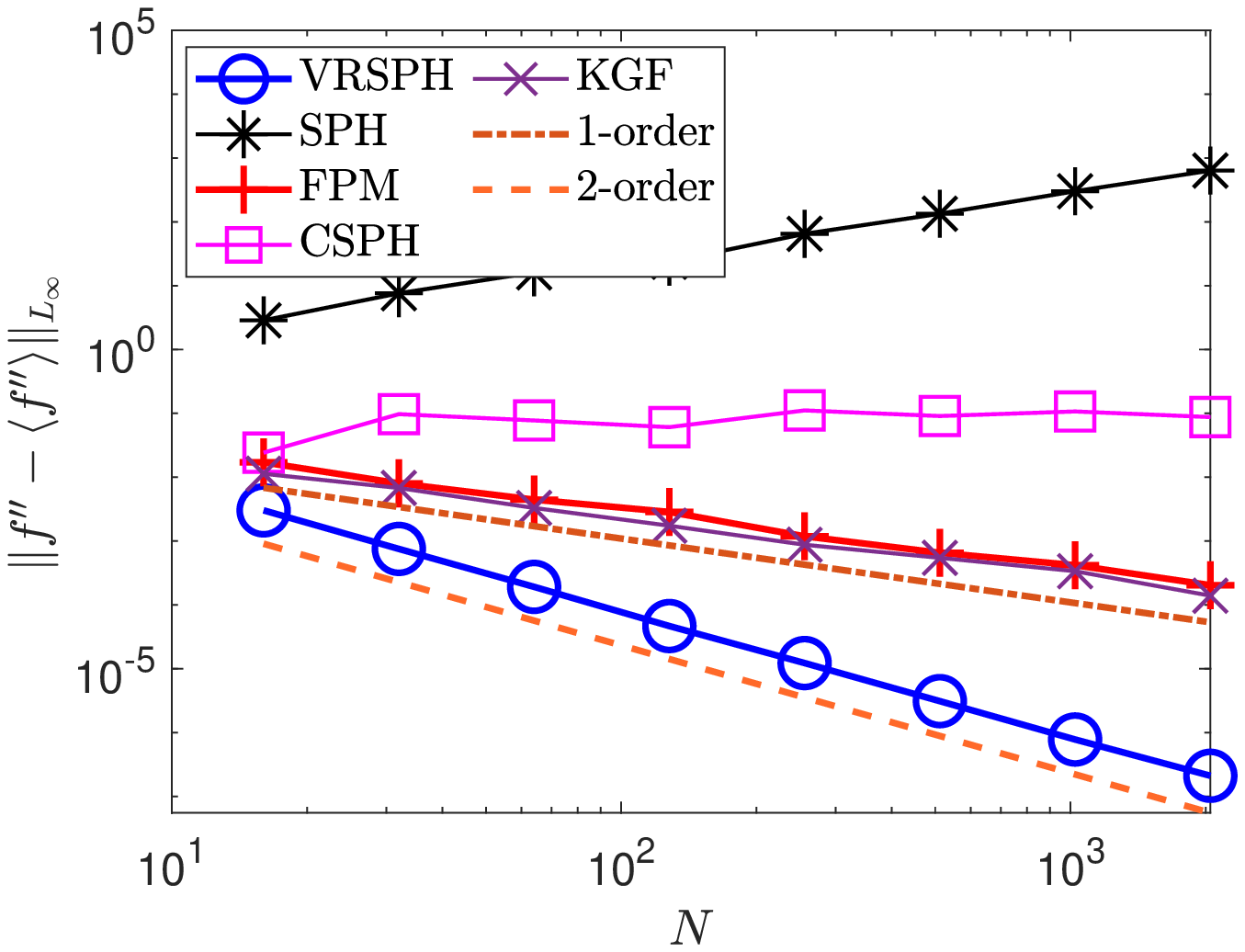}
  \includegraphics[width=0.3\linewidth]{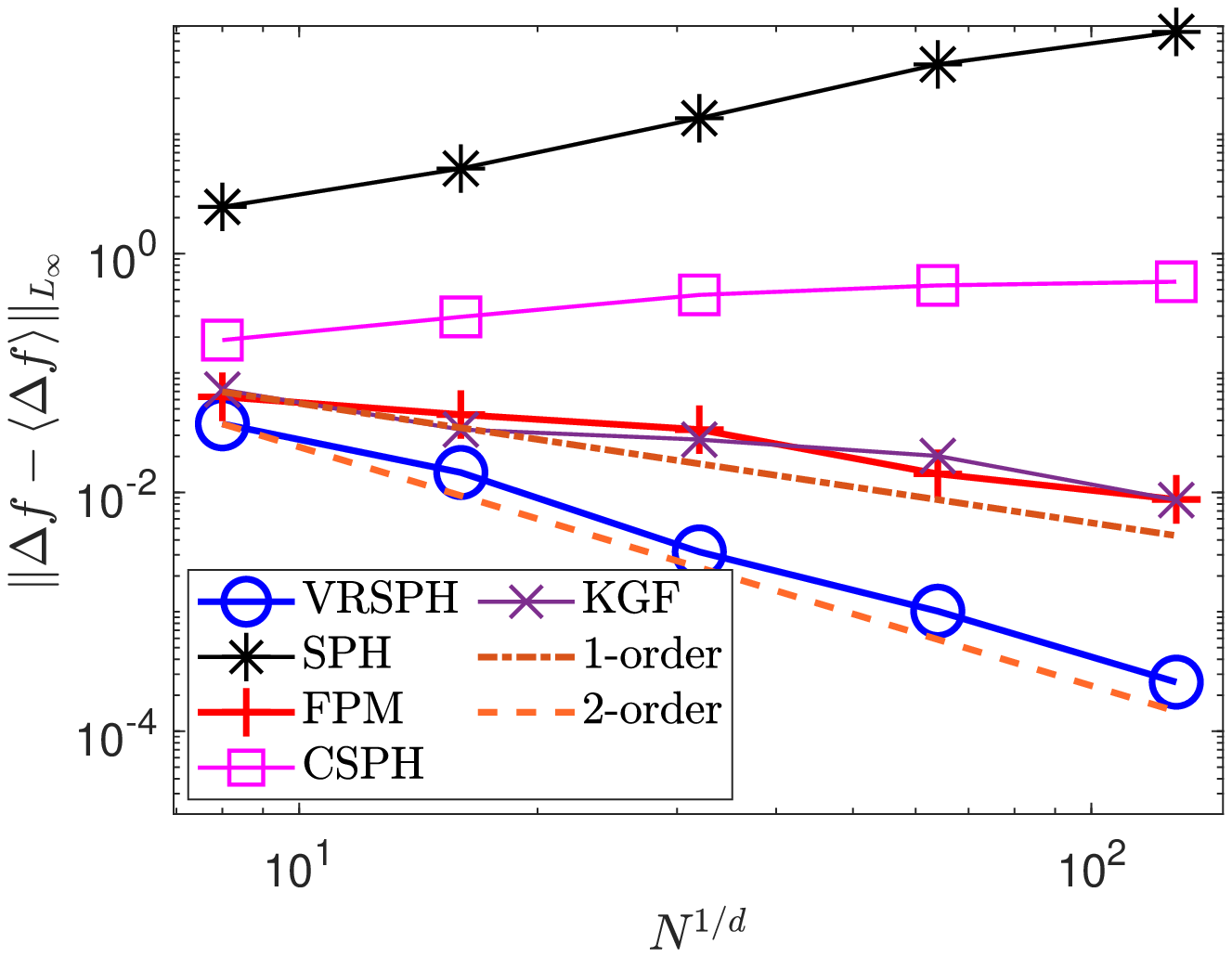}
  \includegraphics[width=0.3\linewidth]{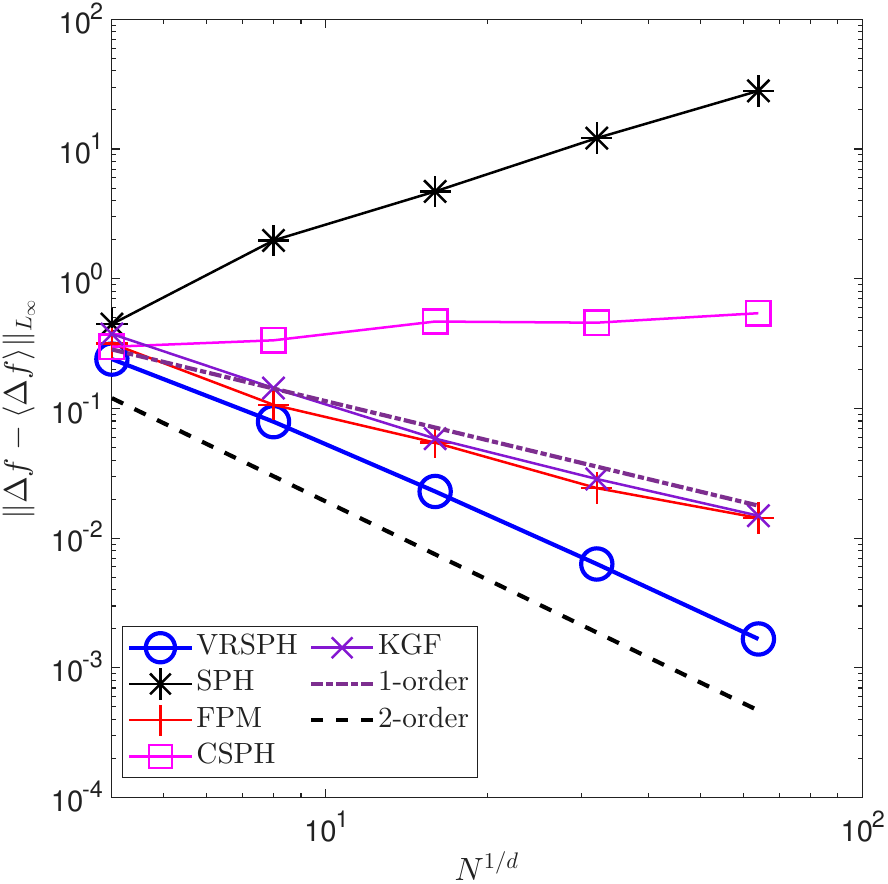}
  \caption{Truncation error of the Laplacian in 1D (left), 2D (middle) and 3D (right) }
\end{figure}
Figure \ref{Laplace1D2D} illustrates the errors in approximating the Laplace operator in the presence of randomly perturbed particle distributions. It is evident that the traditional SPH method exhibits divergence, and the CSPH method fails to ensure first-order accuracy in truncation error. Similarly, the VRSPH method is capable of enhancing the Laplace operator approximation error to second-order accuracy.
\begin{figure}[!h]
\label{irrMat}
  \centering
  \includegraphics[width=0.4\linewidth,height=4.2cm]{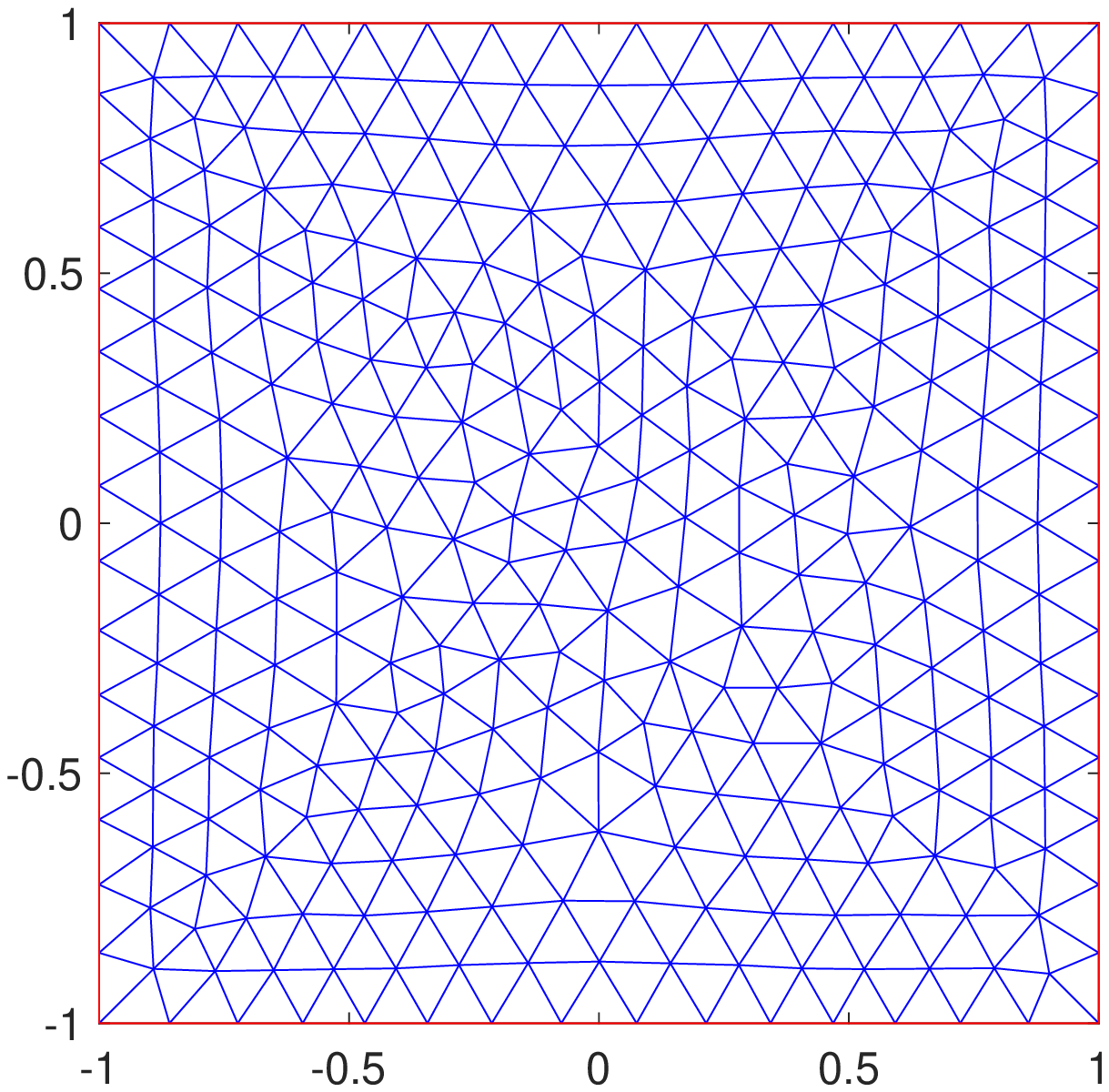}
  \includegraphics[width=0.4\linewidth,height=4.2cm]{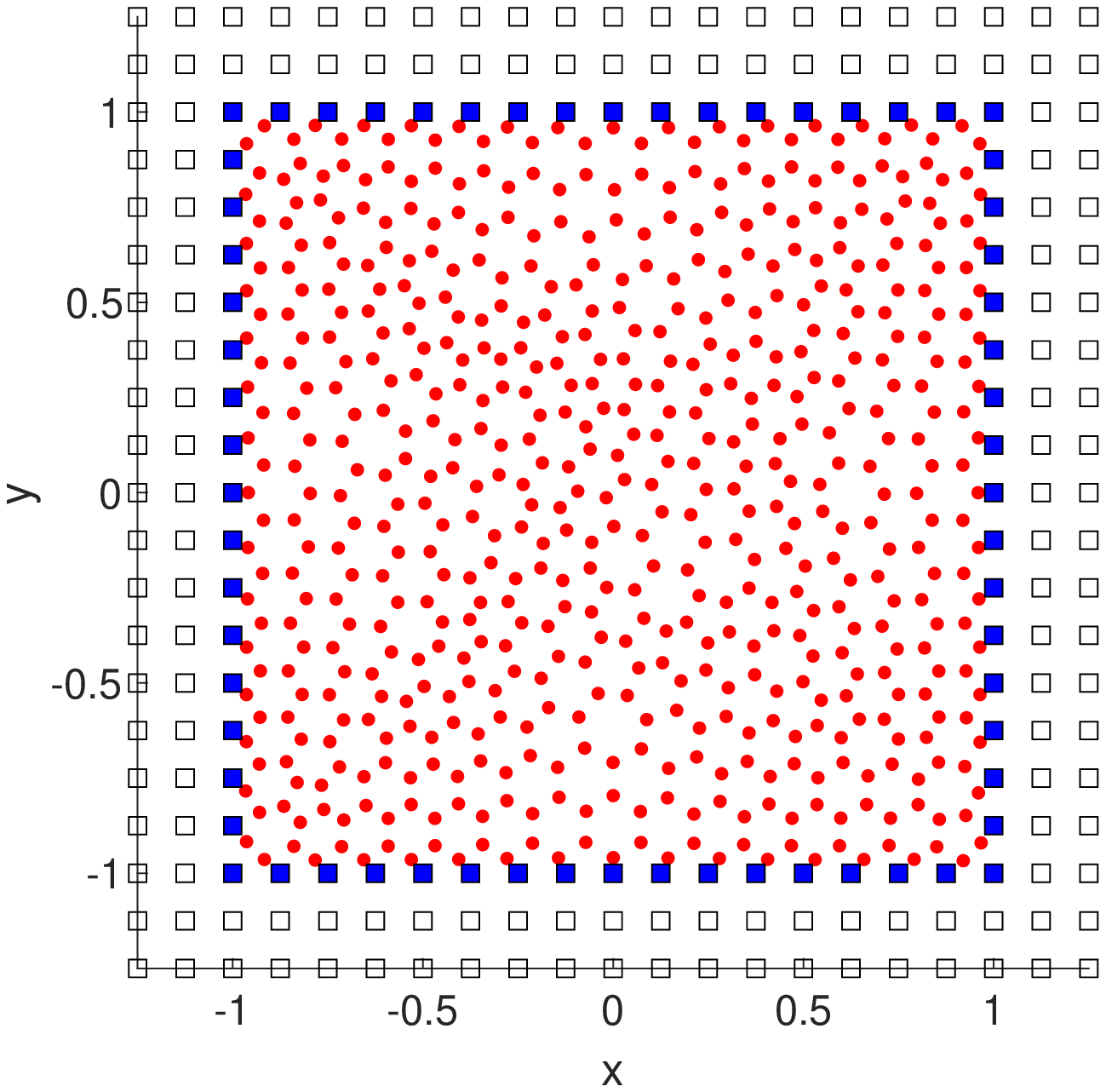}
  \caption{Triangular element (left) and corresponding particle (right) distributions in 2D.}
\end{figure}

\begin{example} As depicted in Figure \ref{irrMat}, we have discretized the domain $[-1,1]^2$ using a triangular mesh with a maximum edge length of $\Delta x$, which was generated utilizing the MATLAB Partial Differential Equation Toolbox \cite{Matlab}. The particles are positioned at the centroids of the triangular elements. Similar irregular distribution  have also been explored in previous studies \cite{liu_restoring_2006,zhang2018decoupled}. In this case, we have set $h=4\Delta x$. \end{example}

\begin{figure} [!h]
\label{Merror} 
\centering 
\includegraphics[width=0.4\linewidth]{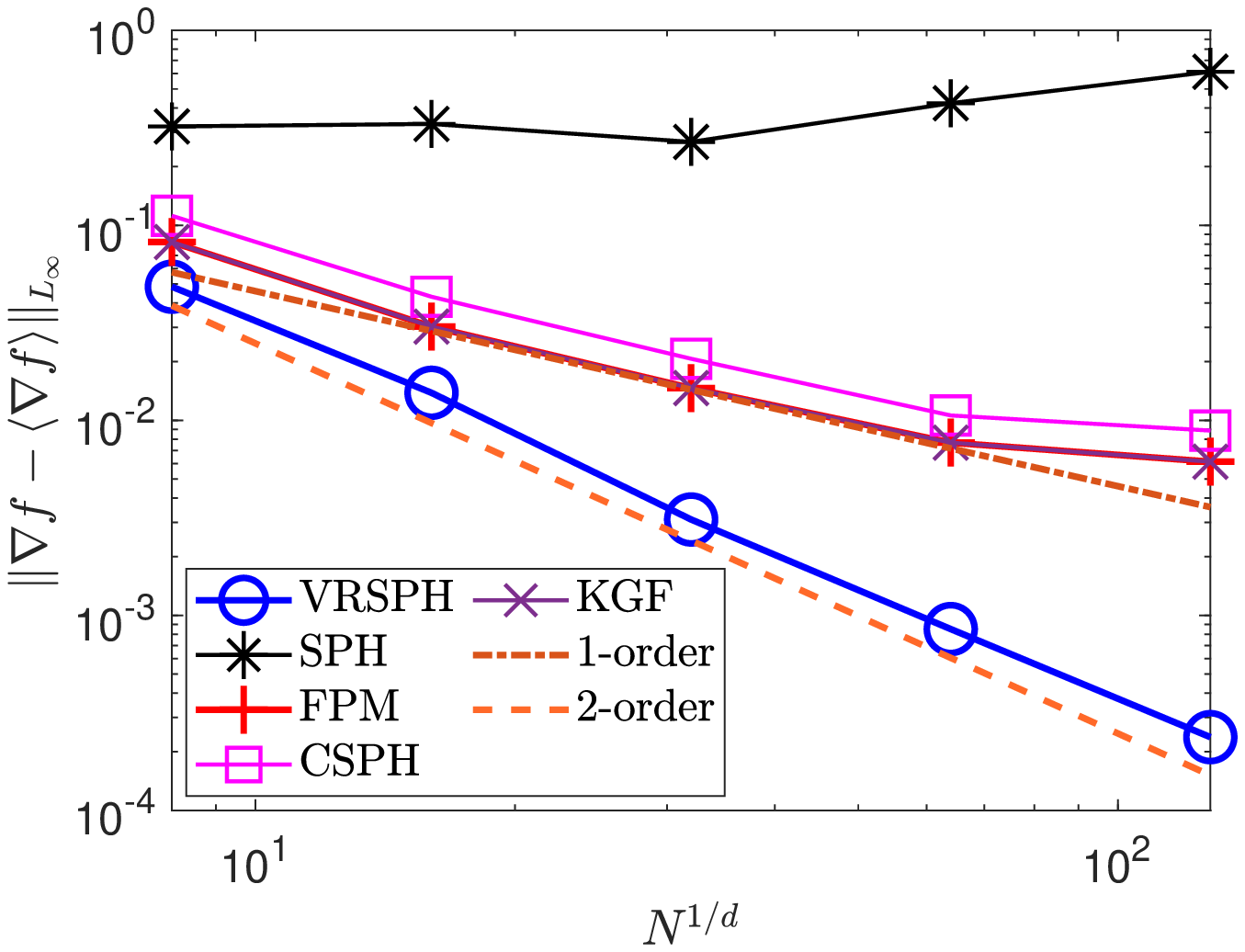} 
\includegraphics[width=0.4\linewidth]{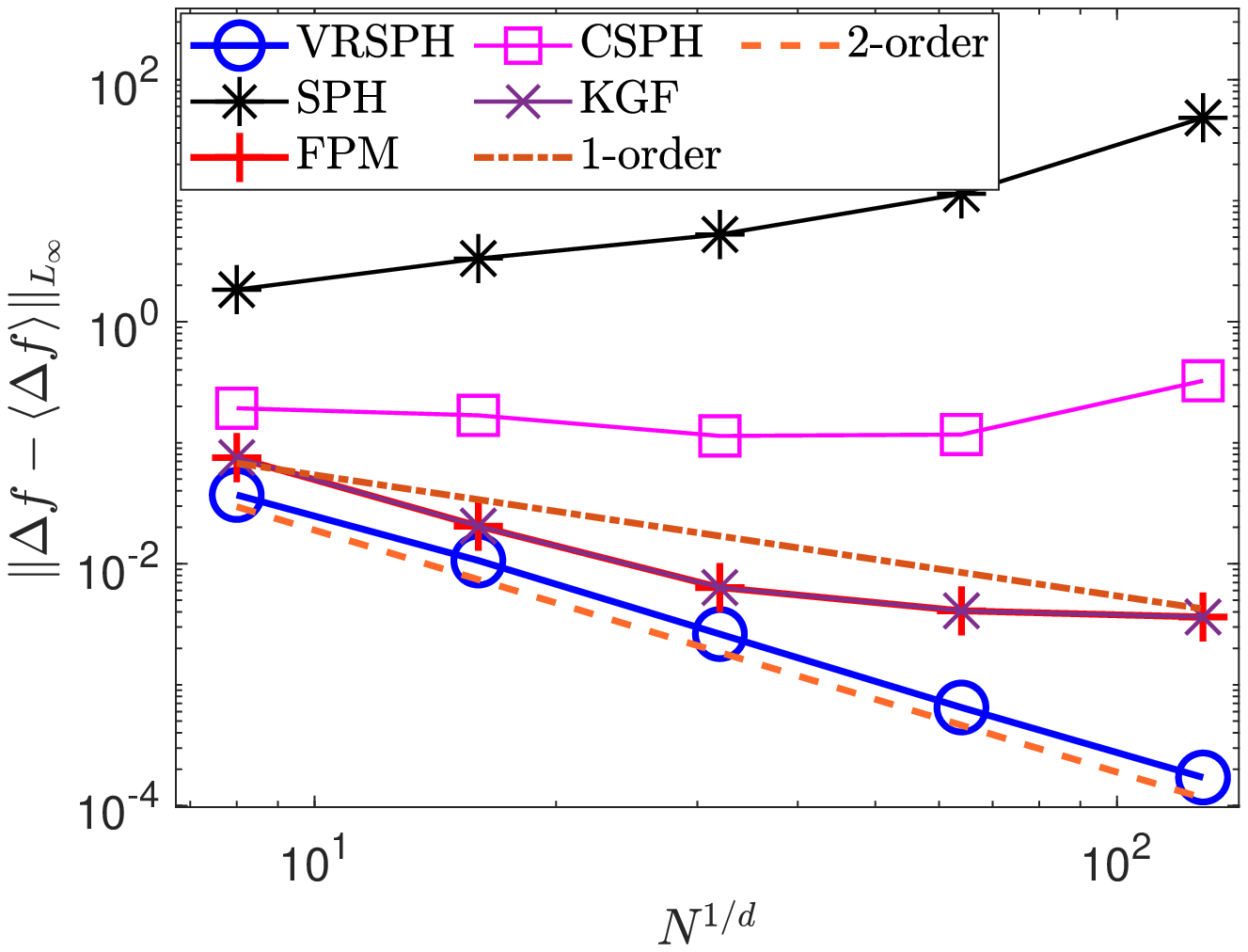} 
\caption{Gradient (left) and Laplace (right) with irregular distributed particles in 2D}
\end{figure}
Figure \ref{Merror} shows the particle approximation results of five different methods for the gradient and Laplace operators. It can be seen that the FPM method barely maintains first-order accuracy in truncation error. In contrast, VRSPH still stably maintains second-order accuracy.
\subsection{Approximation accuracy of boundary particles}
It is important to note that the kernel function is truncated at the boundary, and the kernel approximation of boundary particles cannot maintain second-order accuracy. Therefore, it is crucial to improve the approximation of the function values and derivative values of boundary particles. In this section, we will remove the virtual particles outside the boundary and study the approximation effect of the VRSPH method on boundary particles.

\begin{example}
We investigated the truncation error of boundary points when particles are uniformly distributed. Specifically, this refers to particles being uniformly distributed without placing virtual particles outside the boundary. Here, we set $h=3\Delta x$.
\end{example}

Figure \ref{uniformNonuniform} shows the particle approximation for Example 4. 
\begin{figure}[!h]
\label{uniformNonuniform}
  \centering
  \includegraphics[width=0.4\linewidth]{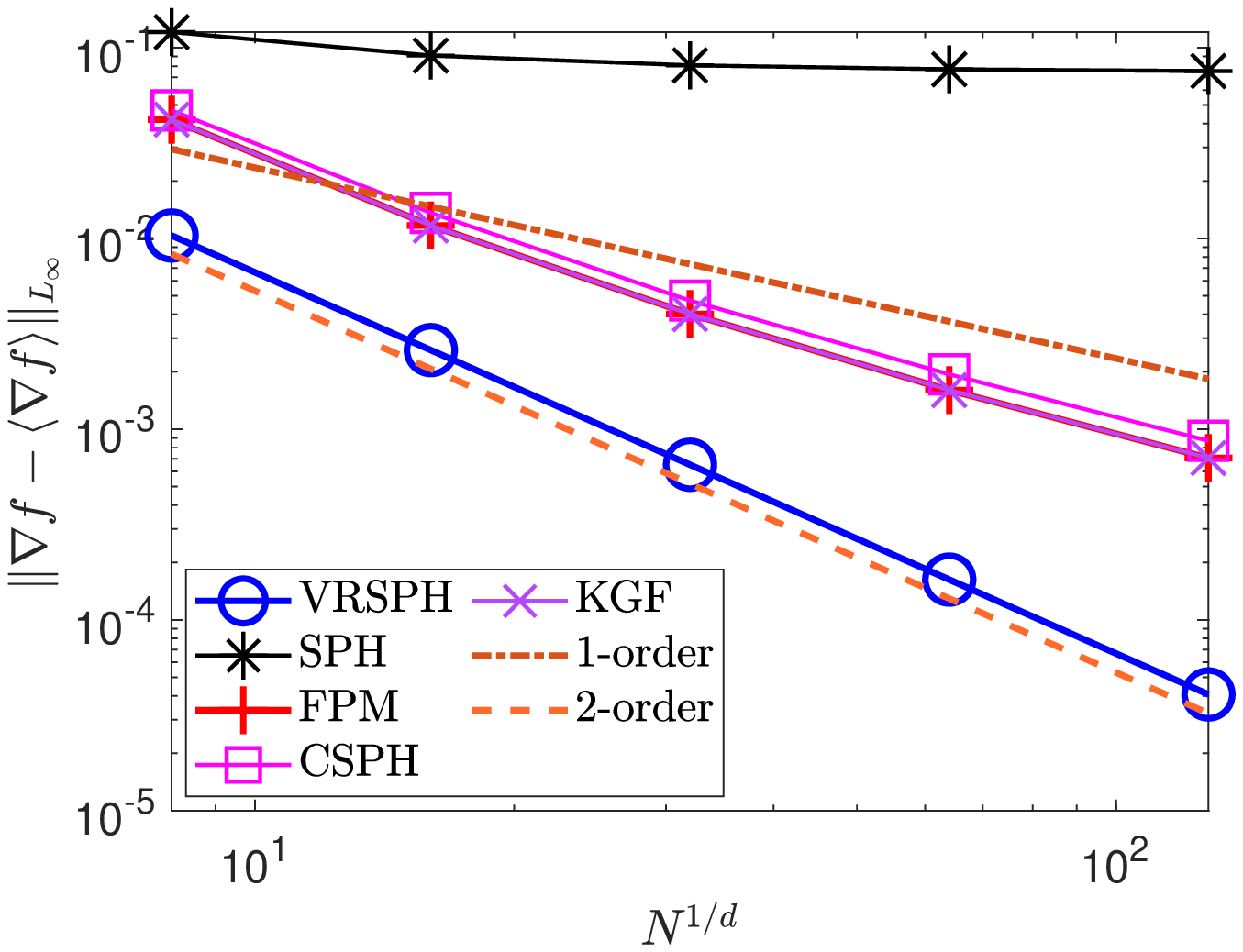}
  \includegraphics[width=0.4\linewidth]{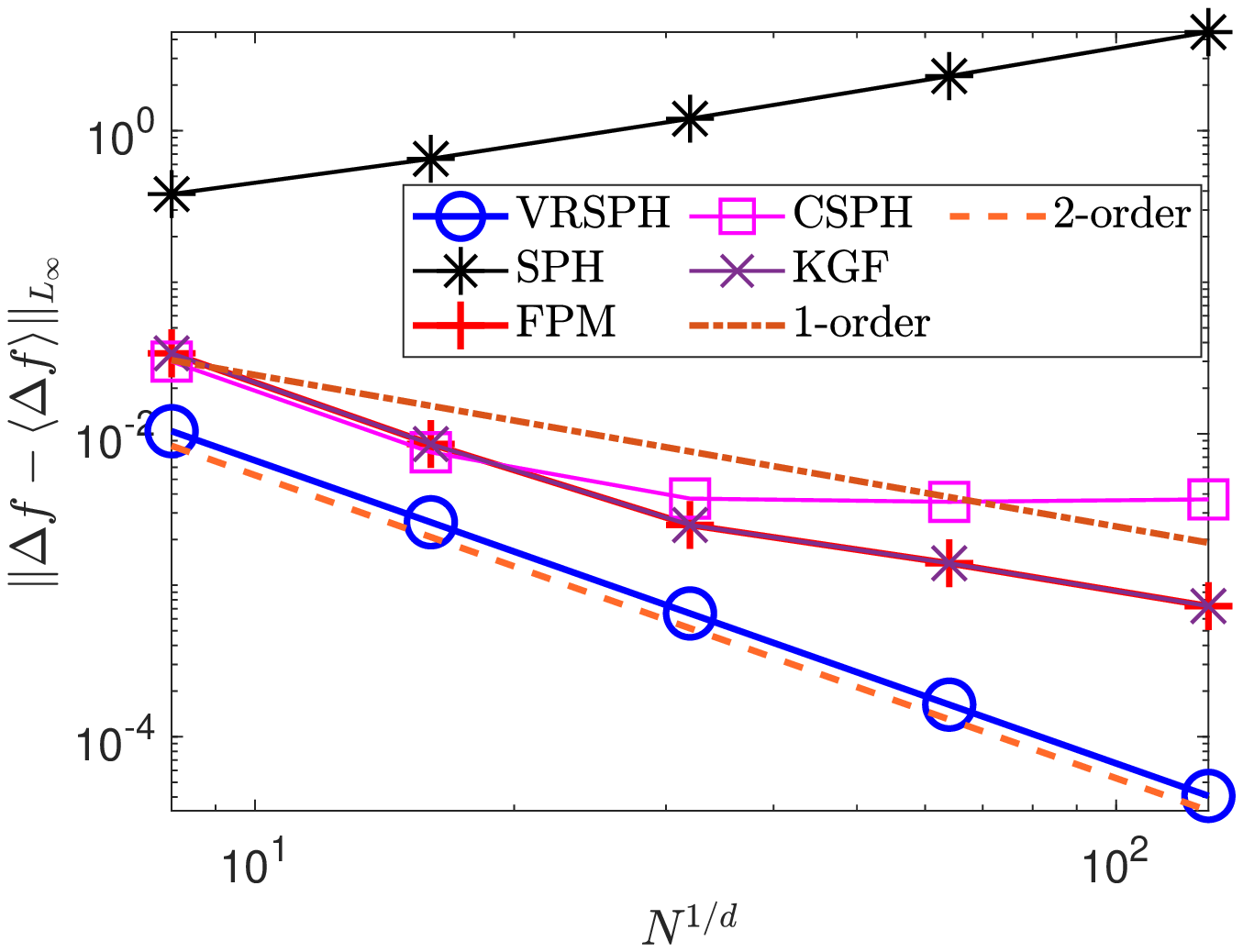}
  \caption{Gradient (left) and Laplace (right) with uniform distribution without virtual particles in 2D}
\end{figure}
It can be observed that the VRSPH method still maintains second-order accuracy, even though the kernel approximation is only first-order accurate at boundary points. 
It can be inferred that, if particles near the boundary are evenly distributed while those in the interior experience random perturbations, VRSPH can still ensure second-order accuracy in approximating gradients and Laplace operators. However, VRSPH does not guarantee second-order accuracy for boundary particles under irregular distributions. In fact, if particles near the boundary undergo random perturbations, and with $h=3\Delta x$, the truncation error no longer maintains second-order accuracy.
\subsection{CPU test}
In order to test the computational efficiency of different methods, we investigated the relationship between computational expense (CPU time) and numerical accuracy for each methods. We recorded the CPU time for calculating the Laplace derivative in 2D ten times, with a particle number of $N=2^n, n=3,4,5,6,7$. The ratio of the influence radius to the average particle spacing $h/\Delta x$ is set to 3. The CPU consumption times and error are show in Figure \ref{CPUtest}. 
It can be observed that VRSPH requires significantly less CPU time than other methods to achieve same accuracy. Note that the VRSPH method necessitates an  estimation of $r_N$ to determine the smoothing length $h$. The CPU time of VRSPH presented in Figure \ref{CPUtest} already incorporates the computational cost for estimating $r_N$. Furthermore, we compared the total computational time of VRSPH with the time required specifically for calculating $r_N$ across different numbers of particles, as illustrated in Figure \ref{CPUtestrN}. 
\begin{figure}[!h] 
  \centering
  \includegraphics[width=0.32\linewidth]{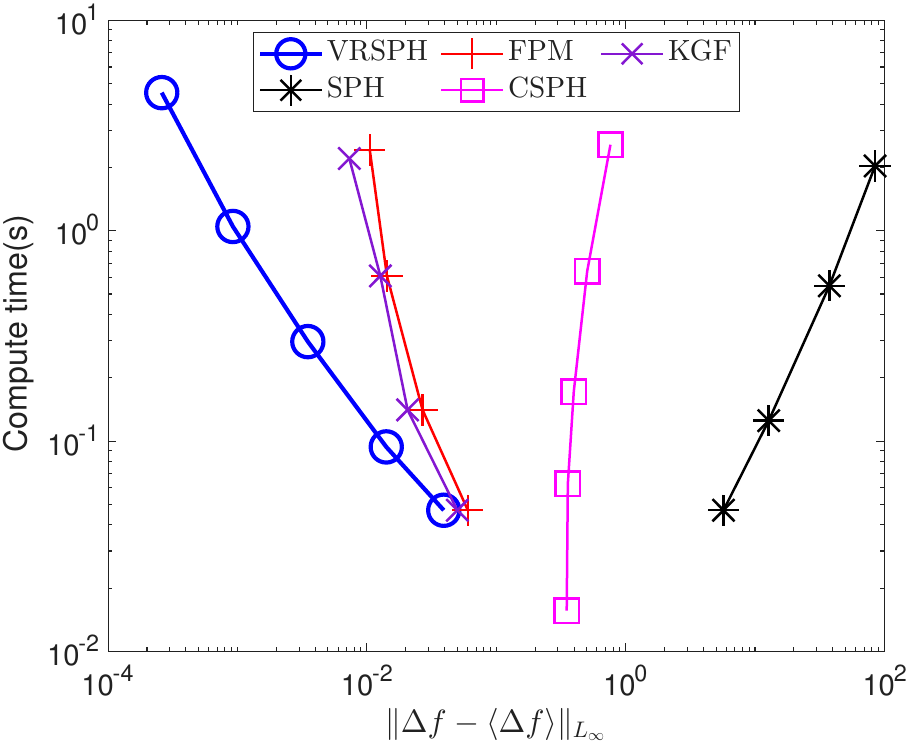}
  \includegraphics[width=0.32\linewidth]{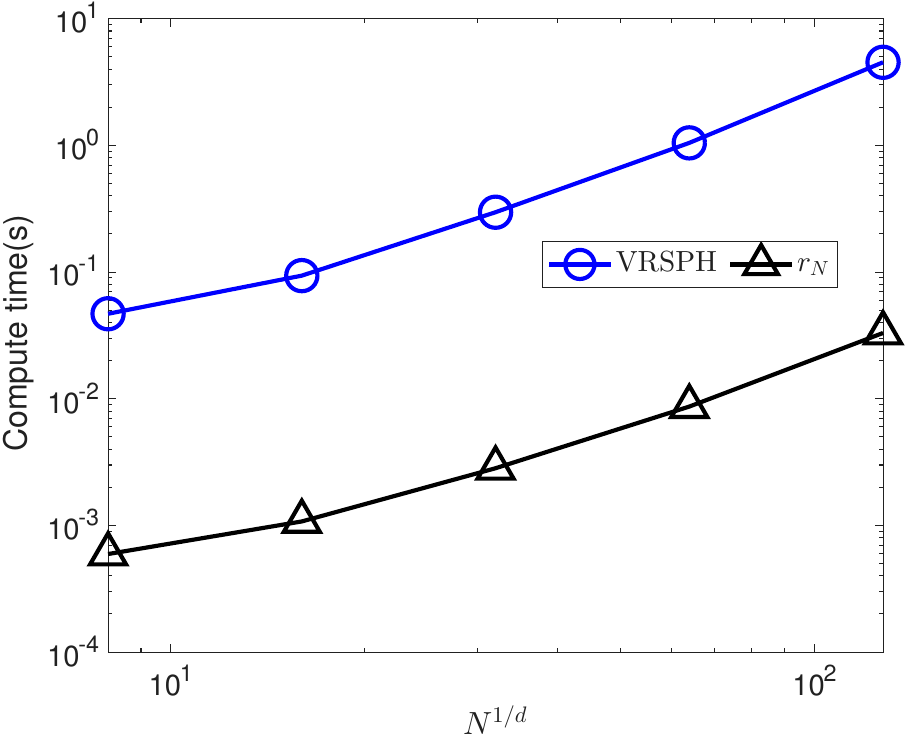} 
  \caption{Compute time for $h/\Delta x=3$ in 2D.}
  \label{CPUtest}
\end{figure}
\begin{figure}[!h] 
  \centering
  \includegraphics[width=0.32\linewidth]{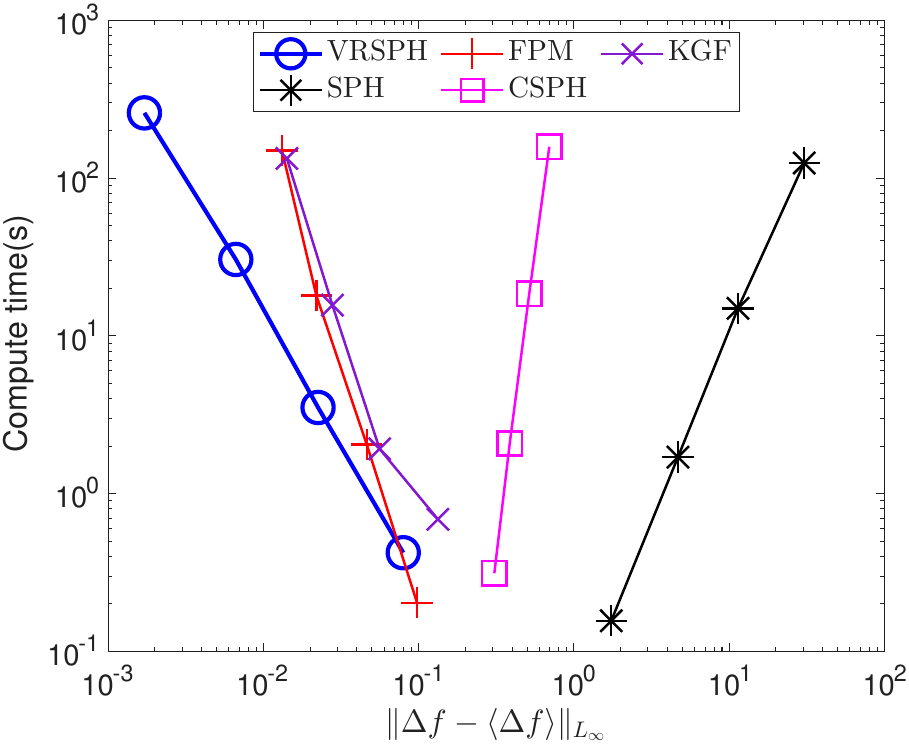}
  \includegraphics[width=0.32\linewidth]{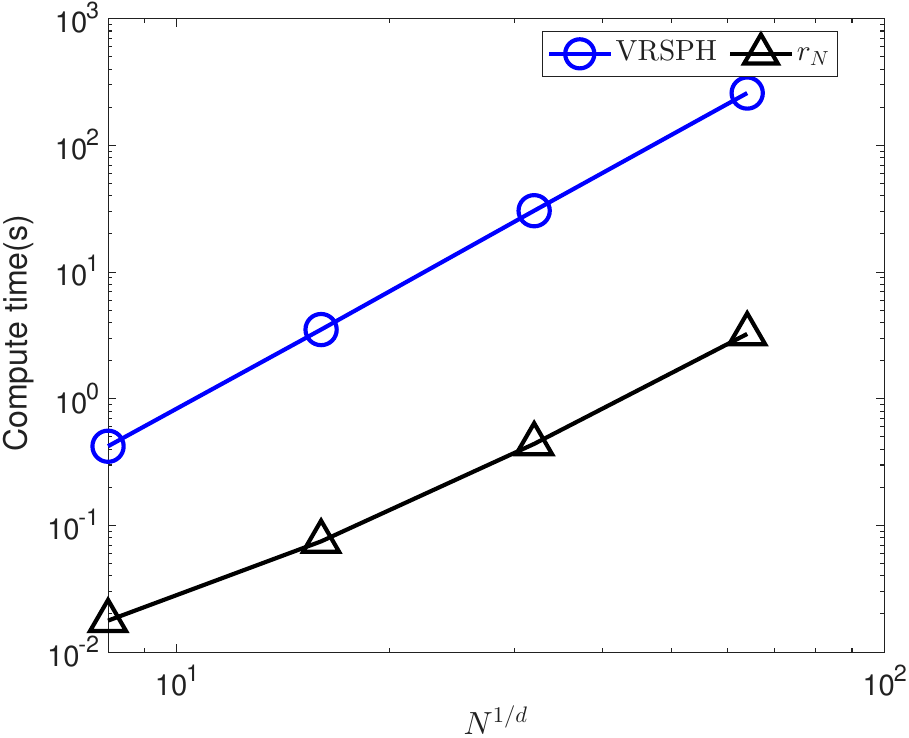}
  \caption{Compute time for $h/\Delta x=3$ in 3D }
    \label{CPUtestrN}
\end{figure}
The results demonstrate that the computational cost for $r_N$ accounts for approximately $1\%$ (for 2D) or $5\%$ (for 3D) of the total algorithm's execution time.

\subsection{ Variable coefficient Poisson equation}
We will test numerical scheme \eqref{scheme} with randomly perturbed particles and uniform particles.
\begin{example}
\textbf{Accuray test.} We now test the convergence rates of the proposed schemes \eqref{scheme}. Let $\Omega=[0,1]^2$ , and we substitute
\begin{equation*}
  \left\{
\begin{array}l
u=\sin(\pi x)\sin(\pi y)\\
a=x^2+y^2+1
  \end{array}
\right.
\end{equation*}
to \eqref{eq1.1} to get $f(x,y)$ at the right hand side of the equation.
Here, we set $ h=3 \Delta x$.
\end{example}

The results for randomly perturbed particles and uniform particles are listed in Table \ref{tab:foo} and Table \ref{tab:foo2}, respectively.  
\begin{table}[!h]
{\footnotesize
  \caption{Errors and convergence rates with randomly perturbed particles}  \label{tab:foo}
\begin{center}
\begin{tabular}{ccccc}\hline
\multirow{2}{*}{ $N^{1/d}$ } & \multicolumn{2}{c}{SPH} & \multicolumn{2}{c}{VRSPH} \\ \cline{2-5}
&$\|u(\mathbf{x})-U\|_\infty$ & Order&$\|u(\mathbf{x})-U\|_\infty$ & Order \\
\hline
20&1.2604e-01&--&5.7539e-03&--\\
40&9.5327e-02&0.402&1.3800e-03&2.060\\
80&9.3375e-02&0.029&3.4862e-04&1.985\\
160&9.3170e-02&0.003&8.7993e-05&1.986\\
\hline
\end{tabular}
\end{center}
}
\end{table}
\begin{table}[!h]
{\footnotesize
  \caption{Errors and convergence rates with uniform particles}  \label{tab:foo2}
\begin{center}
\begin{tabular}{ccccc}\hline
\multirow{2}{*}{ $N^{1/d}$ } & \multicolumn{2}{c}{SPH} & \multicolumn{2}{c}{VRSPH} \\ \cline{2-5}
&$\|u(\mathbf{x})-U\|_\infty$ & Order&$\|u(\mathbf{x})-U\|_\infty$ & Order \\
\hline
20&3.8755e-02&--&7.0421e-03&--\\
40&1.9578e-02&0.9851&1.7623e-03&1.999\\
80&9.9280e-03&0.9797&4.4023e-04&2.001\\
160&5.0103e-03&0.9866&1.1005e-04&2.000\\
\hline
\end{tabular}
\end{center}
}
\end{table}
It can be observed that the convergence order of the SPH method remains first-order under uniform particle distribution. This confirms that the accuracy of particle approximation drops by one order due to the kernel function truncated by boundary. From Table \ref{tab:foo}, it is evident that the convergence rate of the SPH method approaches zero when affected by irregular particle distributions. Conversely, the VRSPH demonstrates second-order convergence in scenarios involving both uniform and irregular particle distributions, thereby confirming the validity of Theorem \ref{convergence}.

\subsection{Poiseuille flow}
The flow of fluid between two parallel, stationary plates at $y = 0$ and $y = l$ is described by Poiseuille flow. 
The experimental setup  in this study are similar to those outlined in references \cite{liu_modeling_2005,morris1997modeling}. Initially at rest, the fluid is propelled by a body force $F$  until it attains steady-state conditions. For this investigation, we have chosen the parameters: the separation between plates $l = 10^{-3}$ m, fluid density $\rho = 10^3$ kg/m$^3$  and kinematic viscosity $\nu = 10^{-6}$ m$^2$/s.  The body force $F$ influences the maximum fluid velocity $v_0$ and subsequently the Reynolds number. We have explored three scenarios with distinct body forces: (i)~$F = 8 \times 10^{-4}$~m/s$^2$, (ii)~$F = 4$~m/s$^2$, and (iii)~$F = 16$~m/s$^2$. The corresponding peak velocities are $v_0 = 1.0 \times 10^{-4}$~m/s, $0.5$~m/s, and $2$~m/s, leading to Reynolds numbers of $\text{Re} = 0.1$, $500$, and $2000$, respectively. Since the onset of turbulence occurs at \text{Re} $>$ 2300, flow regimes beyond this threshold are not considered in the present analysis. The speed of sound was set at $c = 20$ m/s for all cases. 
The simulation is Lagrangian and   the average particle spacing is set as \(\Delta x=  5\times10^{-5}\)m.  A second-order
leapfrog scheme is used for time integration and the time step was set to $10^{-4}$ s. 
Figure \ref{Poiseuille} showcases a comparison between the instantaneous velocity profiles and analytical solutions. The maximum relative errors at the final steady state are below  $7.63\times10^{-5}$. 
The results demonstrate that the proposed method remains stable and effective within the tested range of \text{Re} = 0.1 to 2000.
These findings underscore the adeptness of the VRSPH method in discretizing the governing equations for fluid dynamics, as evidenced by its close alignment with theoretical predictions. 
\begin{figure} [!h]
\label{Poiseuille} 
\centering  
\includegraphics[width=0.29\linewidth]{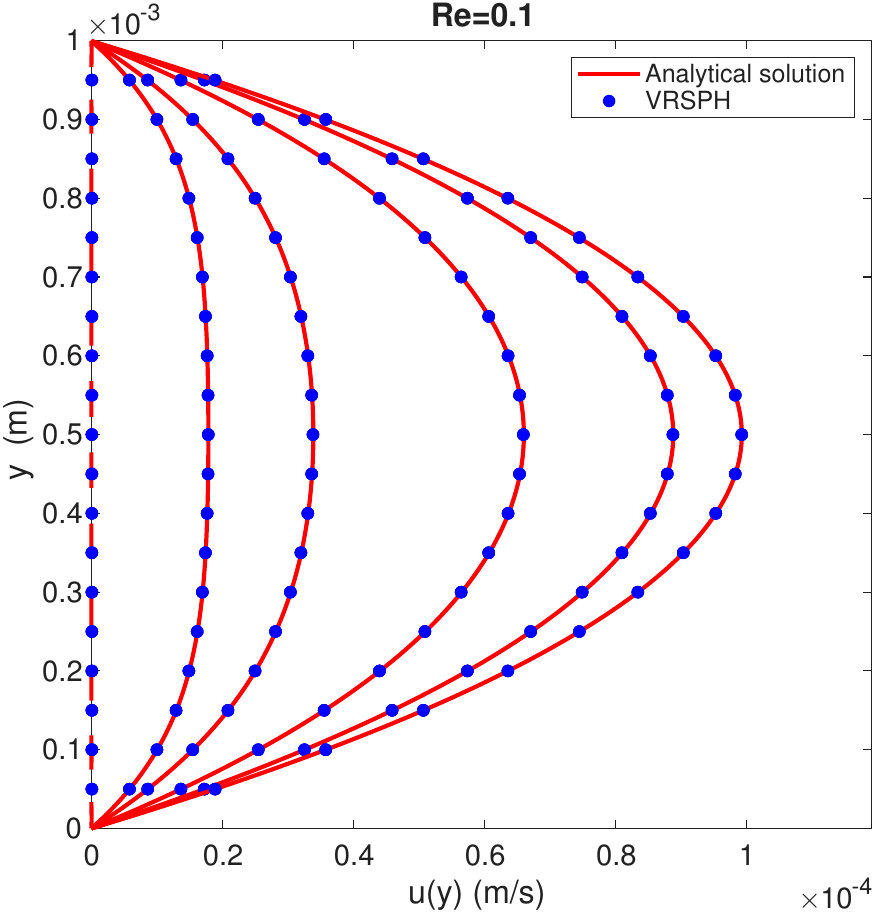} 
\includegraphics[width=0.29\linewidth]{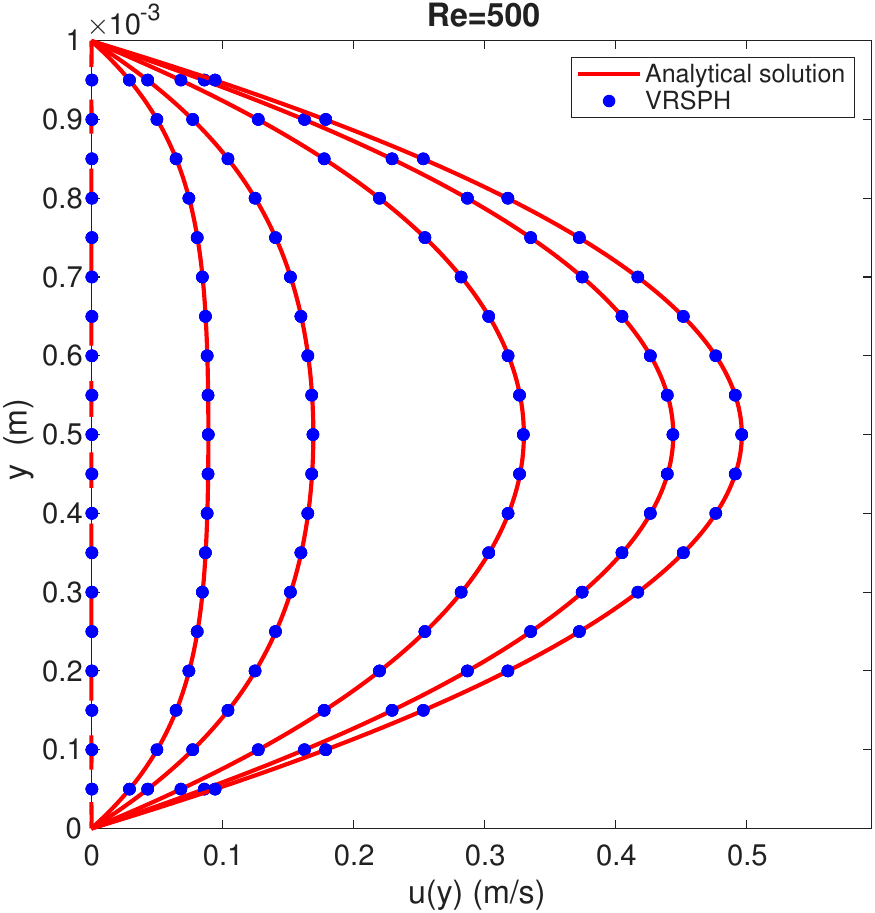}
\includegraphics[width=0.29\linewidth]{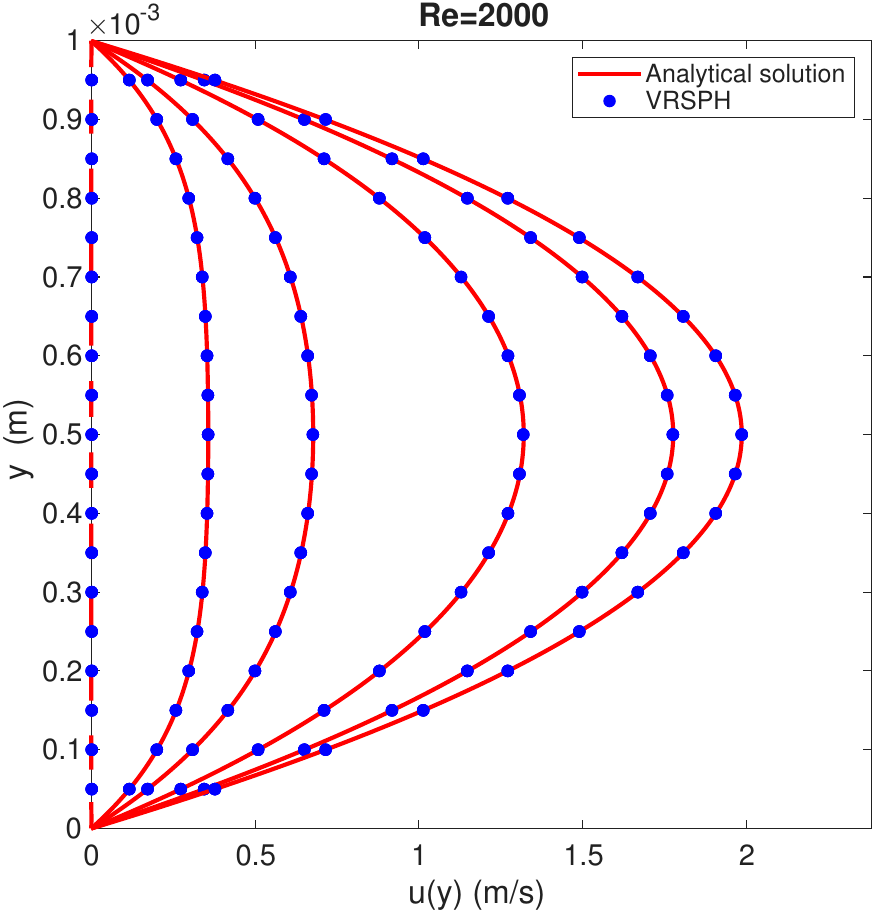}
\caption{Comparison of analytical solutions for transient velocities under different body forces ($F = 8 \times 10^{-4}\ \text{m/s}^2$ (left);  
$F = 4 $ m/s$^2$ (middle); $F = 16 $ m/s$^2$ (right). The six curves in each subplot (from left to right)  correspond to time instances $t = 0$, $0.0225$, $0.045$, $0.1125$, $0.225$, and $0.5$ seconds, respectively.}
\end{figure}

We also examine the error convergence for this Poiseuille flow case. As shown in Table \ref{Tab:3}, three different spatial resolutions $\Delta x$ are used. For time integration, a second-order leapfrog scheme is employed with a sufficiently small, fixed time step so that spatial errors remain dominant. We quantify the error between the numerical solution $u(\mathbf{y})$ and analytical solution $U_{ana}(\mathbf{y})$ in the infinity norm (i.e., $\|u(\mathbf{y})-U_{ana}(\mathbf{y})\|_\infty$), where $\mathbf{y}$ represents the 
y-coordinate of representative particles sampled from the channel flow.  It is evident that the spatial error converges at second-order accuracy, which aligns with the theoretical prediction. 
\begin{table}[!h] 
{\footnotesize
  \caption{Errors and convergence rates at $T=0.0225$s with VRSPH} 
  \label{Tab:3}
\begin{center}
\begin{tabular}{cccccccc}\hline 
\multirow{2}{*}{ $\Delta x$ }&\multirow{2}{*}{ $\Delta t$ } & \multicolumn{2}{c}{Re=0.1} & \multicolumn{2}{c}{Re=500}& \multicolumn{2}{c}{Re=2000} \\ \cline{3-8}
&&Error & Order&Error & Order&Error & Order \\ 
\hline
5.00e-05&1.67e-05&4.0352e-08&--    &2.0177e-04&--&8.0705e-04&--\\
2.50e-05&1.67e-05&1.0115e-08&1.9961&5.0557e-05&1.9967&2.0236e-04&1.9957\\
1.25e-05&1.67e-05&2.5339e-09&1.9971&1.0325e-05&2.2918&4.9424e-05&2.0336\\ 
\hline
\end{tabular}
\end{center}
}
\end{table} 

\section{Conclusions}\label{Sec7}
We have provided regularity conditions for achieving second-order accuracy in particle approximation of gradients and the Laplace operator. Based on this, we have designed a volume reconstruction SPH method to handle situations where particles are irregularly distributed. Compared to the traditional FPM method, the new method achieves higher accuracy in particle approximation of gradients and the Laplace operator by one order with similar CPU time consumption.
Furthermore, we have proven that the method satisfies the discrete maximum principle. Based on this, we have provided an error analysis for the variable coefficient Poisson equation, achieving second-order accuracy. This further enhances the theoretical analysis framework of the SPH method. 

Building upon its demonstrated superiority in high-order operator approximation under irregular particle distributions, the VRSPH method presents a promising approach for high-resolution simulations, such as solid deformation and fluid cavitation. More significantly, by formally analyzing the regularity condition—a sufficient condition for maintaining high-order accuracy—this work also provides a theoretical foundation for developing efficient adaptive SPH algorithms.


\bibliographystyle{unsrt}
\bibliography{BIB}
\end{document}